\documentclass{article}

\usepackage{tikz}
\usepackage{amsmath}
\usepackage{amsfonts}
\usepackage{csquotes}
\usepackage{amssymb}
\usepackage{amsthm}
\usepackage{enumerate}
\usepackage{dsfont}
\usepackage{graphicx}
\usepackage{authblk}
\usepackage{hyperref}
\usepackage{epstopdf}
\usepackage{theoremref}
\usetikzlibrary{arrows,decorations.pathmorphing,backgrounds,positioning,fit}
\usetikzlibrary{positioning}
\usepackage{tkz-graph}
\tikzstyle{vertex}=[circle, draw, inner sep=0pt, minimum size=6pt]

\newtheorem{definition}{Definition}
\newtheorem{thm}{Theorem}

\newtheorem{lemma}{Lemma}[section]

\newtheorem{fact}{Fact}

\newtheorem{example}{Example}

\newcommand{\Z}{\mathbb Z}

\newcommand{\binomial}{\genfrac{[}{]}{0pt}{}}
\newcommand{\NCC}{\text{NCC}}

\title{\textbf{Cyclic Sieving of Matchings}}
\author[1]{Qingzhong Liang \thanks{qzliang@gatech.edu}} 
\author[2]{Grant Bowling \thanks{gbowling@usc.edu}} 
\affil[1]{Department of Mathematics, Georgia Institute of Technology}
\affil[2]{Department of Mathematics, University of Southern California}
 
\date{}

\usepackage{geometry}
\begin{document}
	\maketitle
	
	\begin{abstract}
		The cyclic sieving phenomenon (CSP) was introduced by Reiner, Stanton, and White to study combinatorial structures with actions of cyclic groups. The crucial step is to find a polynomial, for example a q-analog, that satisfies the CSP conditions for an action. This polynomial will give us a lot of information about the symmetry and structure of the set under the action. In this paper, we study the cyclic sieving phenomenon of the cyclic group $C_{2n}$ acting on $P_{n,k}$, which is the set of matchings of $2n$ points on a circle with $k$ crossings. The noncrossing matchings ($k=0$) was recently studied as a Catalan object. In this paper, we study more general cases, the matchings with more number of crossings. We prove that there exists $q$-analog polynomials $f_{n,k}(q)$ such that $(P_{n,k},f_{n,k},C_{2n})$ exhibits the cyclic sieving phenomenon for $k=1,2,3$. In the proof, we also introduce an efficient representation of the elements in $P_{n,k}$, which helps us to understand the symmetrical structure of the set.
	\end{abstract}
	
	\section{Introduction}
	The cyclic sieving phenomenon was first introduced by Reiner, Stanton, and White in \cite{CSP}. 
	
	\begin{definition}
		Let $X$ be a finite set, and $C_n$ be a cyclic group of order $n$ acting on $X$. Fix an isomorphism $\omega$ from $C_n$ to the $n$th roots of unity. Furthermore, let $f(q)$ be a polynomial in $q$. Then, the triple $(X,f(q),C_n)$ is said to exhibit the $\emph{cyclic sieving phenomenon (CSP)}$ if
		
		$$f(\omega(c))=|\{x\in X|c(x)=x\}|, \text{ for every } c\in C_n.$$ 
	\end{definition}
	
	The cyclic sieving phenomenon has been established for many interesting classes of objects. For example, the noncrossing partition of $n$-gon established in \cite{CSP}, and the annular noncrossing permutation established in \cite{kim}.
	
	Motivated by the recent results of CSP, we look at $P_n$, the set of matchings of $2n$ points on a circle. The structure of $P_n$ was studied by Lam in \cite{lam} as a poset. In this paper, we study the CSP on each level of the poset. For a matching $\tau\in P_n$, let $c(\tau)$ denote the number of crossings in $\tau$. Let $P_{n,k}=\{\tau\in P_n|c(\tau)=k\}$ be the subset of $P_n$ with $k$ crossings ($0\leqslant k\leqslant \frac{n(n+1)}{2}$). Consider the operation $\sigma_{2n}$ acting on the set $P_{n}$ by $\tau\mapsto \sigma_{2n}(\tau)$, where
	
	\begin{equation}
	\sigma_{2n}(\tau)(i)=
	\left\{\begin{array}{ll}
	\tau(2n)+1 & \text{ if } i=1\\
	\tau(i-1)+1\pmod{2n} & \text{ if } i\neq 1 
	\end{array}\right. .
	\end{equation}
	
	The following is an example of $\sigma_{2n}$ acting on a element $\tau$ in $P_{7,1}$, which is the set of matchings of $14$ points on a circle with one crossing. Assume $\tau$ is the following matching.
	
	
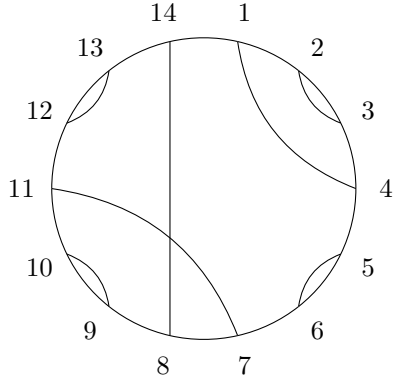
\begin{center}
		\begin{tikzpicture}
		
		\def \n {14}
		\def \radius {2cm}
		
		\foreach \s in {\n,...,1}
		{
			\node at ({-360/\n * (\s-4)}:1.2*\radius) {$\s$} ;
			\draw[-, >=latex] ({360/\n * (\s - 1)}:\radius) 
			arc ({360/\n * (\s -1)}:{360/\n * (\s)}:\radius);
		}
		
		\foreach \from/\to in {2/3, 1/4, 6/7, 9/10,8/12, 13/14}
		\draw ({360/\n * (\from - 1)}:\radius) to [bend left=30] ({360/\n * (\to - 1)}:\radius);
		
		\foreach \from/\to in {5/11}
		\draw ({360/\n * (\from - 1)}:\radius) -- ({360/\n * (\to - 1)}:\radius);
		
		\end{tikzpicture}
	\end{center}
	
	After $\sigma_{2n}$ acted on the matching $\tau$, the edges are rotated clockwise by one unit and we get $\sigma_{2n}(\tau)$ as the following. For example, since $\tau(1)=4$ in the original matching $\tau$, so we have $\sigma_{2n}(\tau)(2)=\tau(2-1)+1=4+1=5$ in the new matching $\sigma_{2n}(\tau)$.
	
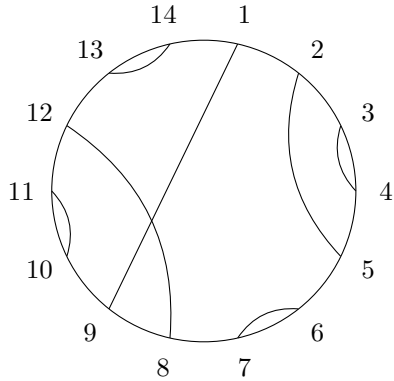
	\begin{center}
		\begin{tikzpicture}
		
		\def \n {14}
		\def \radius {2cm}
		
		\foreach \s in {\n,...,1}
		{
			\node at ({-360/\n * (\s-4)}:1.2*\radius) {$\s$} ;
			\draw[-, >=latex] ({360/\n * (\s - 1)}:\radius) 
			arc ({360/\n * (\s -1)}:{360/\n * (\s)}:\radius);
		}
		
		\foreach \from/\to in {1/2, 14/3, 5/6, 8/9, 7/11, 12/13}
		\draw ({360/\n * (\from - 1)}:\radius) to [bend left=30] ({360/\n * (\to - 1)}:\radius);
		
		\foreach \from/\to in {4/10}
		\draw ({360/\n * (\from - 1)}:\radius) -- ({360/\n * (\to - 1)}:\radius);
		
		\end{tikzpicture}
	\end{center}
	
	The rotation $\sigma_{2n}$ generates the cyclic group $C_{2n}$ with order $2n$. Let $C_{2n}$ act on the poset $P_n$, and we study the CSP on each level of $P_{n}$.
	
	In section 2, we state the main result of this paper. We prove that for $k=1,2,3$ (i.e. the matchings with one, two, and three crossings), there exists $q$-analog polynomials $f_{n,k}(q)$ such that the triples $(P_{n,k},f_{n,k}(q),C_{2n})$ exhibits the cyclic sieving phenomenon. In \cite{partitions}, Bergerson, Miller, Pliml, Reiner, Shearer, Stanton, and Switala studied one-crossing partitions. The method we use to represent the one-crossing mathchings in this paper and their way to count the one-crossing partitions are similar. We construct a bijection between $P_{n,1}$, the set of one-crossing matchings of $2n$ points, and the noncrossing partial matchings of $2n-4$ points. The authors of \cite{partitions} used a bijection between one crossing partitions of $[n]$ and noncrossing partitions of $[n]$ with $4$ special elements. We use our representation method to establish the CSP of one crossing mathcings $P_{n,1}$. Furthermore, we generalize our method in one crossing matchings, and apply it to the matchings with more number of crossings. Specially, we establish the CSP for two-crossing matchings and three-crossing mathchings of $2n$ points ($P_{n,2}$ and $P_{n,3}$).
	
	We provide our proof in section 3. Before the proof, we also introduce a method called Noncrossing Construction (NCC), which serves as an efficient representation of the elements in $P_{n,k}$ and will be used many times in our proof. 
	
	\section{Main Results}
	\subsection{One-crossing Matchings $P_{n,1}$}
	
	We prove via a bijective counting argument that $|P_{n,1}|=\binom{2n}{n-2}$. Let $f_{n,1}(q)$ be the following q-analog:
	
	\begin{equation}
	f_{n,1}(q)=\binomial{2n}{n-2}_q=\frac{[2n]_q[2n-1]_q\cdots [n+3]_q}{[n-2]_q[n-3]_q\cdots [1]_q}.
	\end{equation}
	
	By the properties of $q$-analog, we have $f_{n,1}(1)=\binomial{2n}{n-2}_{q=1}=\binom{2n}{n-2}=|P_{n,1}|$, which is necessary for the CSP. We study this polynomial further and prove it satisfies the CSP conditions. 
	
	\subsection{Two-crossing Matchings $P_{n,2}$ and Three-crossing Matchings $P_{n,3}$}
	
	We also found the number of matchings of $P_{n,2}$ and $P_{n,3}$ by some combinatorial arguments. The idea is similar to the one-crossing case but requires further classification of the types of crossings. The results are the following.
	
	$$|P_{n,2}|=\frac{n+3}{2}\binom{2n}{n-3},$$ 
	
	$$|P_{n,3}|=\frac{1}{3}\binom{n+5}{2}\binom{2n}{n-4}+\binom{2n}{n-3}.$$
	
	Again, let $f_{n,2}(q)$ and $f_{n,3}(q)$ be the following q-analogs:
	
	$$f_{n,2}(q)=\frac{[n+3]_q}{[2]_q}\binomial{2n}{n-3}_q=\frac{[n+3]_q}{[2]_q}\cdot \frac{[2n]_q[2n-1]_q\cdots [n+4]_q}{[n-3]_q[n-4]_q\cdots [1]_q},$$
	
	$$f_{n,3}(q)=\frac{1}{[3]_q} \binomial{n+5}{2}_q \binomial{2n}{n-4}_q+\binomial{2n}{n-3}_q.$$
	
	\subsection{Cyclic Sieving Phenomenon}
	We prove the q-analog expressions $f_{n,k}(q)$ exhibits the cyclic sieving phenomenon for the operation of $\sigma_{2n}$ on the set $P_{n,k}$ for $k=1,2,3$.
	
	\begin{thm}\thlabel{main}
		Let $\xi_{2n}=e^{\frac{\pi i}{n}}$, which is a primitive root of the equation $x^{2n}-1=0$, then $f_{n,k}(\xi_{2n}^j)=|\{\tau\in P_{n,k}|\sigma_{2n}^j(\tau)=\tau\}|$ for $k=1,2,3$. In other words, $(P_{n,k},f_{n,k}(q),C_{2n})$ exhibits the cyclic sieving phenomenon for $k=1,2,3$.
	\end{thm}

	\section{Proof of the Main Results}
	
	\subsection{Noncrossing Construction ($\NCC$)}
	Before we start to prove the one-crossing case, we first introduce an important idea throughout this paper. The idea is motivated by the proof of $|P_{n,1}|=\binom{2n}{n-2}$. We want to find a bijection between choosing $(n-2)$-element subsets of $2n$ elements and matchings in $P_{n,1}$. 
	
	To achieve this goal, we first define a general algorithm $\textbf{NCC}(n,S)$(Noncrossing Construction), which takes an integer $n\in\Z_{>0}$ and a subset $S\subseteq \{1,2,\cdots,2n\}$ such that $|S|\leqslant n$ as inputs, and returns the corresponding noncrossing partial matching. It turns out that applying this algorithm to $|S|=n-2$ for each fixed $n$ allows us to construct the bijection between $(n-2)$-element subsets of $2n$ elements and matchings in $P_{n,1}$, based on which we establish the CSP of $P_{n,1}$. In the later sections, we will need to apply this algorithm to more general cases ($0\leqslant |S|\leqslant n$), in order to establish the CSP of $P_{n,2}$ and $P_{n,3}$. 
	
	\noindent\makebox[\linewidth]{\rule{\textwidth}{0.4pt}}	
	
	\underline{$\NCC(n,S):$}
	
	\textbf{Input} An integer number $n$, and a subset $S\subseteq\{1,2,\dots,2n\}$ such that $|S|=k\leqslant n$.
	
	\textbf{Progress} Consider $2n$ points, which are labeled in clockwise order on a circle as $\{1,\dots, 2n\}$. Look at each point $i\in S$. First, if the point $i+1 \pmod n$ is not in $S$, then we match $i$ with $i+1$. Then for each unmatched point $i\in S$, we look at the point $i+3$. We match $i$ with $i+3$ if $i+3\notin S$ and $i+3$ has not been matched yet. We continue this process, next trying to match the unmatched points $i$ in $S$ with $i+5$, increasing the target point each time by 2. The end result in this process will be that $2(n-k)$ of the points will be paired, which gives a noncrossing partial matching with $2k$ points unmatched.
	
	\textbf{Output} The noncrossing partial matching with $n-k$ pairs of matched points and $2k$ points unmatched.
	\noindent\makebox[\linewidth]{\rule{\textwidth}{0.4pt}}
	
	Now, for each fixed $n$, we consider $S\subseteq \{1,2,\dots,2n\}$ and explain how $\NCC(n,S)$ gives us a element in $P_{n,1}$ for each $S$. Note that the end results of $\NCC(n,S)$ will always be a noncrossing partial matching that $2(n-2) = 2n-4$ of the points are paired, and $4$ points are unmatched. Thus, there is exactly one way to connect the remaining four points to get a cross, which gives us an element in $P_{n,1}$.
	
	
	

	
	\begin{example}\thlabel{noncrossingexample}
	We give an example with $n = 7$ and $S = \{1,2,3,9,12\}$. After we chose the $5$-element subset $S\subseteq\{1,2,\cdots,14\}$, we apply the algorithm $\NCC(n,S)$ to figure out $\NCC(7, \{1,2,3,9,12\})$, and then get an one crossing mathching from it. The five points of $S$ are circled in the following graphs.
	
	
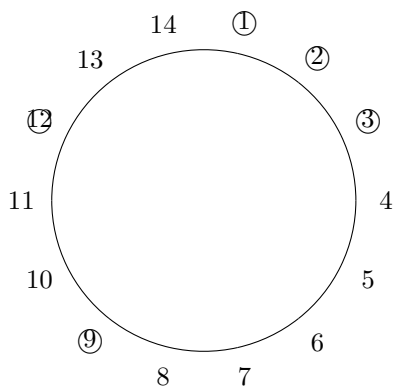
\begin{center}
		\begin{tikzpicture}
		
		\def \n {14}
		\def \radius {2cm}
		
		
		\node at ({-360/\n * (1-4)}:1.2*\radius) {$\textcircled{1}$} ;
		\draw[-, >=latex] ({360/\n * (1 - 1)}:\radius) 
		arc ({360/\n * (1 -1)}:{360/\n * (1)}:\radius);
		
		\node at ({-360/\n * (2-4)}:1.2*\radius) {$\textcircled{2}$} ;
		\draw[-, >=latex] ({360/\n * (2 - 1)}:\radius) 
		arc ({360/\n * (2 -1)}:{360/\n * (2)}:\radius);
		
		\node at ({-360/\n * (3-4)}:1.2*\radius) {$\textcircled{3}$} ;
		\draw[-, >=latex] ({360/\n * (3 - 1)}:\radius) 
		arc ({360/\n * (3 -1)}:{360/\n * (3)}:\radius);
		
		\node at ({-360/\n * (4-4)}:1.2*\radius) {$4$} ;
		\draw[-, >=latex] ({360/\n * (4 - 1)}:\radius) 
		arc ({360/\n * (4 -1)}:{360/\n * (4)}:\radius);
		
		\node at ({-360/\n * (5-4)}:1.2*\radius) {$5$} ;
		\draw[-, >=latex] ({360/\n * (5 - 1)}:\radius) 
		arc ({360/\n * (5 -1)}:{360/\n * (5)}:\radius);
		
		\node at ({-360/\n * (6-4)}:1.2*\radius) {$6$} ;
		\draw[-, >=latex] ({360/\n * (6 - 1)}:\radius) 
		arc ({360/\n * (6 -1)}:{360/\n * (6)}:\radius);
		
		\node at ({-360/\n * (7-4)}:1.2*\radius) {$7$} ;
		\draw[-, >=latex] ({360/\n * (7 - 1)}:\radius) 
		arc ({360/\n * (7 -1)}:{360/\n * (7)}:\radius);
		
		\node at ({-360/\n * (8-4)}:1.2*\radius) {$8$} ;
		\draw[-, >=latex] ({360/\n * (8 - 1)}:\radius) 
		arc ({360/\n * (8 -1)}:{360/\n * (8)}:\radius);
		
		\node at ({-360/\n * (9-4)}:1.2*\radius) {$\textcircled{9}$} ;
		\draw[-, >=latex] ({360/\n * (9 - 1)}:\radius) 
		arc ({360/\n * (9 -1)}:{360/\n * (9)}:\radius);
		
		\node at ({-360/\n * (10-4)}:1.2*\radius) {$10$} ;
		\draw[-, >=latex] ({360/\n * (10 - 1)}:\radius) 
		arc ({360/\n * (10 -1)}:{360/\n * (10)}:\radius);
		
		\node at ({-360/\n * (11-4)}:1.2*\radius) {$11$} ;
		\draw[-, >=latex] ({360/\n * (11 - 1)}:\radius) 
		arc ({360/\n * (11 -1)}:{360/\n * (11)}:\radius);
		
		\node at ({-360/\n * (12-4)}:1.2*\radius) {$\textcircled{12}$} ;
		\draw[-, >=latex] ({360/\n * (12 - 1)}:\radius) 
		arc ({360/\n * (12 -1)}:{360/\n * (12)}:\radius);
		
		\node at ({-360/\n * (13-4)}:1.2*\radius) {$13$} ;
		\draw[-, >=latex] ({360/\n * (13 - 1)}:\radius) 
		arc ({360/\n * (13 -1)}:{360/\n * (13)}:\radius);
		
		\node at ({-360/\n * (14-4)}:1.2*\radius) {$14$} ;
		\draw[-, >=latex] ({360/\n * (14 - 1)}:\radius) 
		arc ({360/\n * (14 -1)}:{360/\n * (14)}:\radius);
		
		\foreach \from/\to in {}
		\draw ({360/\n * (\from - 1)}:\radius) to [bend left=30] ({360/\n * (\to - 1)}:\radius);
		
		\foreach \from/\to in {}
		\draw ({360/\n * (\from - 1)}:\radius) -- ({360/\n * (\to - 1)}:\radius);
		
		\end{tikzpicture}
	\end{center}
	
	First, we look at each $i\in S$ to see whether we can match $i$ with $i+1$. 
	\begin{itemize}
		\item $1\in S$, but $1+1=2$ is also in $S$, so we cannot match $1$ with $2$;
		\item $2\in S$, but $2+1=3$ is also in $S$, so we cannot match $2$ with $3$;
		\item $3\in S$ and $3+1=4\notin S$, so we match $3$ with $4$;
		\item $9\in S$ and $9+1=10\notin S$, so we match $9$ with $10$;
		\item $12\in S$ and $12+1=13\notin S$, so we match $12$ with $13$.
	\end{itemize} 
	
	
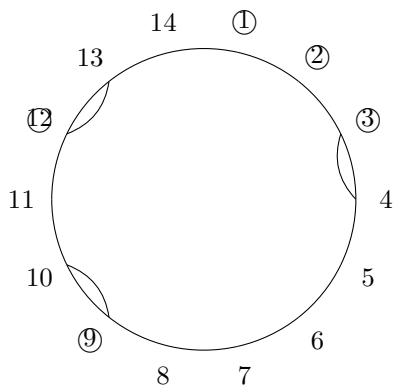
\begin{center}
		\begin{tikzpicture}
		
		\def \n {14}
		\def \radius {2cm}
		
		\node at ({-360/\n * (1-4)}:1.2*\radius) {$\textcircled{1}$} ;
		\draw[-, >=latex] ({360/\n * (1 - 1)}:\radius) 
		arc ({360/\n * (1 -1)}:{360/\n * (1)}:\radius);
		
		\node at ({-360/\n * (2-4)}:1.2*\radius) {$\textcircled{2}$} ;
		\draw[-, >=latex] ({360/\n * (2 - 1)}:\radius) 
		arc ({360/\n * (2 -1)}:{360/\n * (2)}:\radius);
		
		\node at ({-360/\n * (3-4)}:1.2*\radius) {$\textcircled{3}$} ;
		\draw[-, >=latex] ({360/\n * (3 - 1)}:\radius) 
		arc ({360/\n * (3 -1)}:{360/\n * (3)}:\radius);
		
		\node at ({-360/\n * (4-4)}:1.2*\radius) {$4$} ;
		\draw[-, >=latex] ({360/\n * (4 - 1)}:\radius) 
		arc ({360/\n * (4 -1)}:{360/\n * (4)}:\radius);
		
		\node at ({-360/\n * (5-4)}:1.2*\radius) {$5$} ;
		\draw[-, >=latex] ({360/\n * (5 - 1)}:\radius) 
		arc ({360/\n * (5 -1)}:{360/\n * (5)}:\radius);
		
		\node at ({-360/\n * (6-4)}:1.2*\radius) {$6$} ;
		\draw[-, >=latex] ({360/\n * (6 - 1)}:\radius) 
		arc ({360/\n * (6 -1)}:{360/\n * (6)}:\radius);
		
		\node at ({-360/\n * (7-4)}:1.2*\radius) {$7$} ;
		\draw[-, >=latex] ({360/\n * (7 - 1)}:\radius) 
		arc ({360/\n * (7 -1)}:{360/\n * (7)}:\radius);
		
		\node at ({-360/\n * (8-4)}:1.2*\radius) {$8$} ;
		\draw[-, >=latex] ({360/\n * (8 - 1)}:\radius) 
		arc ({360/\n * (8 -1)}:{360/\n * (8)}:\radius);
		
		\node at ({-360/\n * (9-4)}:1.2*\radius) {$\textcircled{9}$} ;
		\draw[-, >=latex] ({360/\n * (9 - 1)}:\radius) 
		arc ({360/\n * (9 -1)}:{360/\n * (9)}:\radius);
		
		\node at ({-360/\n * (10-4)}:1.2*\radius) {$10$} ;
		\draw[-, >=latex] ({360/\n * (10 - 1)}:\radius) 
		arc ({360/\n * (10 -1)}:{360/\n * (10)}:\radius);
		
		\node at ({-360/\n * (11-4)}:1.2*\radius) {$11$} ;
		\draw[-, >=latex] ({360/\n * (11 - 1)}:\radius) 
		arc ({360/\n * (11 -1)}:{360/\n * (11)}:\radius);
		
		\node at ({-360/\n * (12-4)}:1.2*\radius) {$\textcircled{12}$} ;
		\draw[-, >=latex] ({360/\n * (12 - 1)}:\radius) 
		arc ({360/\n * (12 -1)}:{360/\n * (12)}:\radius);
		
		\node at ({-360/\n * (13-4)}:1.2*\radius) {$13$} ;
		\draw[-, >=latex] ({360/\n * (13 - 1)}:\radius) 
		arc ({360/\n * (13 -1)}:{360/\n * (13)}:\radius);
		
		\node at ({-360/\n * (14-4)}:1.2*\radius) {$14$} ;
		\draw[-, >=latex] ({360/\n * (14 - 1)}:\radius) 
		arc ({360/\n * (14 -1)}:{360/\n * (14)}:\radius);
		
		\foreach \from/\to in {1/2, 6/7, 9/10}
		\draw ({360/\n * (\from - 1)}:\radius) to [bend left=30] ({360/\n * (\to - 1)}:\radius);
		
		\foreach \from/\to in {}
		\draw ({360/\n * (\from - 1)}:\radius) -- ({360/\n * (\to - 1)}:\radius);
		
		\end{tikzpicture}
	\end{center}
	
	Now, for each unmatched vertex $i$ in $S$, we look at $i+3$ to see whether we can match $i$ with $i+3$. There are two unmatched vertices in $S$, which are $1$ and $2$.
	\begin{itemize}
		\item $1\in S$, but $1+3=4$ is matched with $3$, so we cannot match $1$ with $4$;
		\item $2\in S$, $2+3=5\notin S$, and $5$ is unmatched, so we can match $2$ with $5$.
	\end{itemize} 
	
	
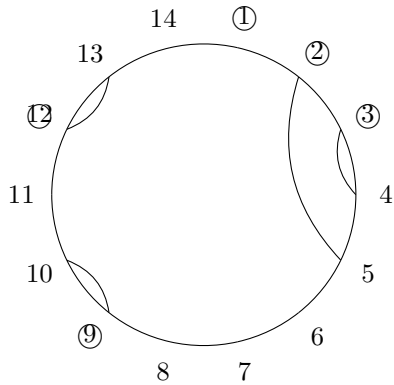
\begin{center}
		\begin{tikzpicture}
		
		\def \n {14}
		\def \radius {2cm}
		
		\node at ({-360/\n * (1-4)}:1.2*\radius) {$\textcircled{1}$} ;
		\draw[-, >=latex] ({360/\n * (1 - 1)}:\radius) 
		arc ({360/\n * (1 -1)}:{360/\n * (1)}:\radius);
		
		\node at ({-360/\n * (2-4)}:1.2*\radius) {$\textcircled{2}$} ;
		\draw[-, >=latex] ({360/\n * (2 - 1)}:\radius) 
		arc ({360/\n * (2 -1)}:{360/\n * (2)}:\radius);
		
		\node at ({-360/\n * (3-4)}:1.2*\radius) {$\textcircled{3}$} ;
		\draw[-, >=latex] ({360/\n * (3 - 1)}:\radius) 
		arc ({360/\n * (3 -1)}:{360/\n * (3)}:\radius);
		
		\node at ({-360/\n * (4-4)}:1.2*\radius) {$4$} ;
		\draw[-, >=latex] ({360/\n * (4 - 1)}:\radius) 
		arc ({360/\n * (4 -1)}:{360/\n * (4)}:\radius);
		
		\node at ({-360/\n * (5-4)}:1.2*\radius) {$5$} ;
		\draw[-, >=latex] ({360/\n * (5 - 1)}:\radius) 
		arc ({360/\n * (5 -1)}:{360/\n * (5)}:\radius);
		
		\node at ({-360/\n * (6-4)}:1.2*\radius) {$6$} ;
		\draw[-, >=latex] ({360/\n * (6 - 1)}:\radius) 
		arc ({360/\n * (6 -1)}:{360/\n * (6)}:\radius);
		
		\node at ({-360/\n * (7-4)}:1.2*\radius) {$7$} ;
		\draw[-, >=latex] ({360/\n * (7 - 1)}:\radius) 
		arc ({360/\n * (7 -1)}:{360/\n * (7)}:\radius);
		
		\node at ({-360/\n * (8-4)}:1.2*\radius) {$8$} ;
		\draw[-, >=latex] ({360/\n * (8 - 1)}:\radius) 
		arc ({360/\n * (8 -1)}:{360/\n * (8)}:\radius);
		
		\node at ({-360/\n * (9-4)}:1.2*\radius) {$\textcircled{9}$} ;
		\draw[-, >=latex] ({360/\n * (9 - 1)}:\radius) 
		arc ({360/\n * (9 -1)}:{360/\n * (9)}:\radius);
		
		\node at ({-360/\n * (10-4)}:1.2*\radius) {$10$} ;
		\draw[-, >=latex] ({360/\n * (10 - 1)}:\radius) 
		arc ({360/\n * (10 -1)}:{360/\n * (10)}:\radius);
		
		\node at ({-360/\n * (11-4)}:1.2*\radius) {$11$} ;
		\draw[-, >=latex] ({360/\n * (11 - 1)}:\radius) 
		arc ({360/\n * (11 -1)}:{360/\n * (11)}:\radius);
		
		\node at ({-360/\n * (12-4)}:1.2*\radius) {$\textcircled{12}$} ;
		\draw[-, >=latex] ({360/\n * (12 - 1)}:\radius) 
		arc ({360/\n * (12 -1)}:{360/\n * (12)}:\radius);
		
		\node at ({-360/\n * (13-4)}:1.2*\radius) {$13$} ;
		\draw[-, >=latex] ({360/\n * (13 - 1)}:\radius) 
		arc ({360/\n * (13 -1)}:{360/\n * (13)}:\radius);
		
		\node at ({-360/\n * (14-4)}:1.2*\radius) {$14$} ;
		\draw[-, >=latex] ({360/\n * (14 - 1)}:\radius) 
		arc ({360/\n * (14 -1)}:{360/\n * (14)}:\radius);
		
		\foreach \from/\to in {1/2, 6/7, 9/10, 14/3}
		\draw ({360/\n * (\from - 1)}:\radius) to [bend left=30] ({360/\n * (\to - 1)}:\radius);
		\end{tikzpicture}
	\end{center}
	
	Now, for each unmatched vertex $i$ in $S$, we look at $i+5$ to see whether we can match $i$ with $i+5$. The only one unmatched vertex in $S$ id $1$.
	\begin{itemize}
		\item $1\in S$, $1+5=6\notin S$, and $6$ is unmatched, so we can match $1$ with $6$.
	\end{itemize} 
	
	
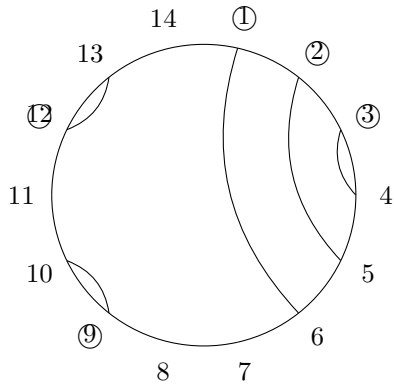
\begin{center}
		\begin{tikzpicture}
		
		\def \n {14}
		\def \radius {2cm}
		
		\node at ({-360/\n * (1-4)}:1.2*\radius) {$\textcircled{1}$} ;
		\draw[-, >=latex] ({360/\n * (1 - 1)}:\radius) 
		arc ({360/\n * (1 -1)}:{360/\n * (1)}:\radius);
		
		\node at ({-360/\n * (2-4)}:1.2*\radius) {$\textcircled{2}$} ;
		\draw[-, >=latex] ({360/\n * (2 - 1)}:\radius) 
		arc ({360/\n * (2 -1)}:{360/\n * (2)}:\radius);
		
		\node at ({-360/\n * (3-4)}:1.2*\radius) {$\textcircled{3}$} ;
		\draw[-, >=latex] ({360/\n * (3 - 1)}:\radius) 
		arc ({360/\n * (3 -1)}:{360/\n * (3)}:\radius);
		
		\node at ({-360/\n * (4-4)}:1.2*\radius) {$4$} ;
		\draw[-, >=latex] ({360/\n * (4 - 1)}:\radius) 
		arc ({360/\n * (4 -1)}:{360/\n * (4)}:\radius);
		
		\node at ({-360/\n * (5-4)}:1.2*\radius) {$5$} ;
		\draw[-, >=latex] ({360/\n * (5 - 1)}:\radius) 
		arc ({360/\n * (5 -1)}:{360/\n * (5)}:\radius);
		
		\node at ({-360/\n * (6-4)}:1.2*\radius) {$6$} ;
		\draw[-, >=latex] ({360/\n * (6 - 1)}:\radius) 
		arc ({360/\n * (6 -1)}:{360/\n * (6)}:\radius);
		
		\node at ({-360/\n * (7-4)}:1.2*\radius) {$7$} ;
		\draw[-, >=latex] ({360/\n * (7 - 1)}:\radius) 
		arc ({360/\n * (7 -1)}:{360/\n * (7)}:\radius);
		
		\node at ({-360/\n * (8-4)}:1.2*\radius) {$8$} ;
		\draw[-, >=latex] ({360/\n * (8 - 1)}:\radius) 
		arc ({360/\n * (8 -1)}:{360/\n * (8)}:\radius);
		
		\node at ({-360/\n * (9-4)}:1.2*\radius) {$\textcircled{9}$} ;
		\draw[-, >=latex] ({360/\n * (9 - 1)}:\radius) 
		arc ({360/\n * (9 -1)}:{360/\n * (9)}:\radius);
		
		\node at ({-360/\n * (10-4)}:1.2*\radius) {$10$} ;
		\draw[-, >=latex] ({360/\n * (10 - 1)}:\radius) 
		arc ({360/\n * (10 -1)}:{360/\n * (10)}:\radius);
		
		\node at ({-360/\n * (11-4)}:1.2*\radius) {$11$} ;
		\draw[-, >=latex] ({360/\n * (11 - 1)}:\radius) 
		arc ({360/\n * (11 -1)}:{360/\n * (11)}:\radius);
		
		\node at ({-360/\n * (12-4)}:1.2*\radius) {$\textcircled{12}$} ;
		\draw[-, >=latex] ({360/\n * (12 - 1)}:\radius) 
		arc ({360/\n * (12 -1)}:{360/\n * (12)}:\radius);
		
		\node at ({-360/\n * (13-4)}:1.2*\radius) {$13$} ;
		\draw[-, >=latex] ({360/\n * (13 - 1)}:\radius) 
		arc ({360/\n * (13 -1)}:{360/\n * (13)}:\radius);
		
		\node at ({-360/\n * (14-4)}:1.2*\radius) {$14$} ;
		\draw[-, >=latex] ({360/\n * (14 - 1)}:\radius) 
		arc ({360/\n * (14 -1)}:{360/\n * (14)}:\radius);
		
		\foreach \from/\to in {1/2, 6/7, 9/10, 14/3, 13/4}
		\draw ({360/\n * (\from - 1)}:\radius) to [bend left=30] ({360/\n * (\to - 1)}:\radius);
		\end{tikzpicture}
	\end{center}
	
	Thus, $\NCC(7,\{1,2,3,9,12\})$ returns the noncrossing partial matching $(1,6)(2,5)(3,4)(9,10)(12,13)$. There are only four points $7,8,11,14$ left, which give us a unique way to get a one-crossing matching $(7,11)(8,14)$ as shown below. So the corresponding element in $P_{n,1}$ is $(1,6)(2,5)(3,4)(9,10)(12,13)(7,11)(8,14)$.
	
	\begin{center}
		\begin{tikzpicture}
		
		\def \n {14}
		\def \radius {2cm}
		
		\node at ({-360/\n * (1-4)}:1.2*\radius) {$\textcircled{1}$} ;
		\draw[-, >=latex] ({360/\n * (1 - 1)}:\radius) 
		arc ({360/\n * (1 -1)}:{360/\n * (1)}:\radius);
		
		\node at ({-360/\n * (2-4)}:1.2*\radius) {$\textcircled{2}$} ;
		\draw[-, >=latex] ({360/\n * (2 - 1)}:\radius) 
		arc ({360/\n * (2 -1)}:{360/\n * (2)}:\radius);
		
		\node at ({-360/\n * (3-4)}:1.2*\radius) {$\textcircled{3}$} ;
		\draw[-, >=latex] ({360/\n * (3 - 1)}:\radius) 
		arc ({360/\n * (3 -1)}:{360/\n * (3)}:\radius);
		
		\node at ({-360/\n * (4-4)}:1.2*\radius) {$4$} ;
		\draw[-, >=latex] ({360/\n * (4 - 1)}:\radius) 
		arc ({360/\n * (4 -1)}:{360/\n * (4)}:\radius);
		
		\node at ({-360/\n * (5-4)}:1.2*\radius) {$5$} ;
		\draw[-, >=latex] ({360/\n * (5 - 1)}:\radius) 
		arc ({360/\n * (5 -1)}:{360/\n * (5)}:\radius);
		
		\node at ({-360/\n * (6-4)}:1.2*\radius) {$6$} ;
		\draw[-, >=latex] ({360/\n * (6 - 1)}:\radius) 
		arc ({360/\n * (6 -1)}:{360/\n * (6)}:\radius);
		
		\node at ({-360/\n * (7-4)}:1.2*\radius) {$7$} ;
		\draw[-, >=latex] ({360/\n * (7 - 1)}:\radius) 
		arc ({360/\n * (7 -1)}:{360/\n * (7)}:\radius);
		
		\node at ({-360/\n * (8-4)}:1.2*\radius) {$8$} ;
		\draw[-, >=latex] ({360/\n * (8 - 1)}:\radius) 
		arc ({360/\n * (8 -1)}:{360/\n * (8)}:\radius);
		
		\node at ({-360/\n * (9-4)}:1.2*\radius) {$\textcircled{9}$} ;
		\draw[-, >=latex] ({360/\n * (9 - 1)}:\radius) 
		arc ({360/\n * (9 -1)}:{360/\n * (9)}:\radius);
		
		\node at ({-360/\n * (10-4)}:1.2*\radius) {$10$} ;
		\draw[-, >=latex] ({360/\n * (10 - 1)}:\radius) 
		arc ({360/\n * (10 -1)}:{360/\n * (10)}:\radius);
		
		\node at ({-360/\n * (11-4)}:1.2*\radius) {$11$} ;
		\draw[-, >=latex] ({360/\n * (11 - 1)}:\radius) 
		arc ({360/\n * (11 -1)}:{360/\n * (11)}:\radius);
		
		\node at ({-360/\n * (12-4)}:1.2*\radius) {$\textcircled{12}$} ;
		\draw[-, >=latex] ({360/\n * (12 - 1)}:\radius) 
		arc ({360/\n * (12 -1)}:{360/\n * (12)}:\radius);
		
		\node at ({-360/\n * (13-4)}:1.2*\radius) {$13$} ;
		\draw[-, >=latex] ({360/\n * (13 - 1)}:\radius) 
		arc ({360/\n * (13 -1)}:{360/\n * (13)}:\radius);
		
		\node at ({-360/\n * (14-4)}:1.2*\radius) {$14$} ;
		\draw[-, >=latex] ({360/\n * (14 - 1)}:\radius) 
		arc ({360/\n * (14 -1)}:{360/\n * (14)}:\radius);
		
		\foreach \from/\to in {1/2, 6/7, 9/10, 14/3, 13/4, 8/12}
		\draw ({360/\n * (\from - 1)}:\radius) to [bend left=30] ({360/\n * (\to - 1)}:\radius);
		
		\foreach \from/\to in {5/11}
		\draw ({360/\n * (\from - 1)}:\radius) -- ({360/\n * (\to - 1)}:\radius);
		
		\end{tikzpicture}
	\end{center}
	\end{example}
	
	\subsection{One-crossing Matchings}
	
	\begin{lemma}\thlabel{bijection}
		$|P_{n,1}|=\binom{2n}{n-2}$.
	\end{lemma}
	
	\begin{proof}
		By applying $\NCC(n,S)$ with $|S|=n-2$, we know that each choice of a $(n-2)$-element subset in $\{1,2,\dots,2n\}$ gives us a unique one crossing matching. 
		
		Now, if we are given a one-crossing matching $\tau\in P_{n,1}$. Assume the crossing is the intersection of the edges $\{b_1,b_3\}$ and $\{b_2,b_4\}$, then $\tau$ can be written in the form of
		
		\begin{equation}
		\tau=(a_1,a_{n+1}) (a_2,a_{n+2}) \dots (a_{n-2}, a_{2n-2}) (b_1, b_3) (b_2,b_4),
		\end{equation}
		
		where each pair $(a_i, a_{i+n})$ represents a non-crossing edge. Note that the two crossing edges $(b_1, b_3)$ and $(b_2,b_4)$ divide the circle in four arcs. Denote the four arcs by $l_{12},l_{23},l_{34},l_{41}$. 
		
		
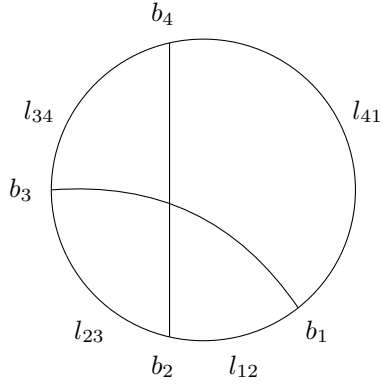
\begin{center}
			\begin{tikzpicture}
			
			\def \n {14}
			\def \radius {2cm}
			
			\node at ({-360/\n * (1-4)}:1.2*\radius) {} ;
			\draw[-, >=latex] ({360/\n * (1 - 1)}:\radius) 
			arc ({360/\n * (1 -1)}:{360/\n * (1)}:\radius);
			
			\node at ({-360/\n * (2-4)}:1.2*\radius) {} ;
			\draw[-, >=latex] ({360/\n * (2 - 1)}:\radius) 
			arc ({360/\n * (2 -1)}:{360/\n * (2)}:\radius);
			
			\node at ({-360/\n * (3-4)}:1.2*\radius) {$l_{41}$} ;
			\draw[-, >=latex] ({360/\n * (3 - 1)}:\radius) 
			arc ({360/\n * (3 -1)}:{360/\n * (3)}:\radius);
			
			\node at ({-360/\n * (4-4)}:1.2*\radius) {} ;
			\draw[-, >=latex] ({360/\n * (4 - 1)}:\radius) 
			arc ({360/\n * (4 -1)}:{360/\n * (4)}:\radius);
			
			\node at ({-360/\n * (5-4)}:1.2*\radius) {} ;
			\draw[-, >=latex] ({360/\n * (5 - 1)}:\radius) 
			arc ({360/\n * (5 -1)}:{360/\n * (5)}:\radius);
			
			\node at ({-360/\n * (6-4)}:1.2*\radius) {$b_1$} ;
			\draw[-, >=latex] ({360/\n * (6 - 1)}:\radius) 
			arc ({360/\n * (6 -1)}:{360/\n * (6)}:\radius);
			
			\node at ({-360/\n * (7-4)}:1.2*\radius) {$l_{12}$} ;
			\draw[-, >=latex] ({360/\n * (7 - 1)}:\radius) 
			arc ({360/\n * (7 -1)}:{360/\n * (7)}:\radius);
			
			\node at ({-360/\n * (8-4)}:1.2*\radius) {$b_2$} ;
			\draw[-, >=latex] ({360/\n * (8 - 1)}:\radius) 
			arc ({360/\n * (8 -1)}:{360/\n * (8)}:\radius);
			
			\node at ({-360/\n * (9-4)}:1.2*\radius) {$l_{23}$} ;
			\draw[-, >=latex] ({360/\n * (9 - 1)}:\radius) 
			arc ({360/\n * (9 -1)}:{360/\n * (9)}:\radius);
			
			\node at ({-360/\n * (10-4)}:1.2*\radius) {} ;
			\draw[-, >=latex] ({360/\n * (10 - 1)}:\radius) 
			arc ({360/\n * (10 -1)}:{360/\n * (10)}:\radius);
			
			\node at ({-360/\n * (11-4)}:1.2*\radius) {$b_3$} ;
			\draw[-, >=latex] ({360/\n * (11 - 1)}:\radius) 
			arc ({360/\n * (11 -1)}:{360/\n * (11)}:\radius);
			
			\node at ({-360/\n * (12-4)}:1.2*\radius) {$l_{34}$} ;
			\draw[-, >=latex] ({360/\n * (12 - 1)}:\radius) 
			arc ({360/\n * (12 -1)}:{360/\n * (12)}:\radius);
			
			\node at ({-360/\n * (13-4)}:1.2*\radius) {} ;
			\draw[-, >=latex] ({360/\n * (13 - 1)}:\radius) 
			arc ({360/\n * (13 -1)}:{360/\n * (13)}:\radius);
			
			\node at ({-360/\n * (14-4)}:1.2*\radius) {$b_4$} ;
			\draw[-, >=latex] ({360/\n * (14 - 1)}:\radius) 
			arc ({360/\n * (14 -1)}:{360/\n * (14)}:\radius);
			
			\foreach \from/\to in {8/13}
			\draw ({360/\n * (\from - 1)}:\radius) to [bend left=30] ({360/\n * (\to - 1)}:\radius);
			
			\foreach \from/\to in {5/11}
			\draw ({360/\n * (\from - 1)}:\radius) -- ({360/\n * (\to - 1)}:\radius);
			
			\end{tikzpicture}
		\end{center}
		
		For each noncrossing edge $(a_i,a_{n+1})$, both of the vertices $a_i$ and $a_{n+i}$ must be on the same arc, say $l_{41}$. The clockwise order from $b_4$ to $b_1$ is $b_4,b_4+1,\dots, b_1-1,b_1 \pmod{2n}$. WLOG, let $a_i$ denote the point appears earlier in this order. Then, we pick $a_i$ to represent this edge. Pick such point from each noncrossing edge, then the one-crossing matching corresponds to the $(n-2)$ point subset
		
		\begin{equation}
		S_{\tau}=\{a_1, a_2, \dots, a_{n-2}\}. 
		\end{equation}
		
		For the one-crossing matching in \thref{noncrossingexample}, take the edge $(2,5)$ as an example. Both $2$ and $5$ are on the arc from $14$ to $7$. The clockwise order of points is $14,1,2,3,4,5,6,7$, in which $2$ appears earlier than $5$. So we pick $2$ to represent this edge in $S$.
		
		Thus, there is a bijective relation between the $(n-2)$-element subsets of $\{1, 2,\dots, 2n\}$ and the one-crossing matchings $\tau \in P_{n,1}$.
	\end{proof}
	
	For any $\tau\in P_{n}$, let 
	
	$$d(\tau)=: \text{ the smallest positive integer } j \text{ such that } \sigma_{2n}^j(\tau)=\tau,$$
	
	then $d(\tau)$ must be a divisor of $2n$. The following lemma asserts that there are only three possible values of $d(\tau)$ for $\tau\in P_{n,1}$.
	
	\begin{lemma}\thlabel{dtau}
		For any $\tau\in P_{n,1}$, $d(\tau)\in\{2n,n,\frac{n}{2}\}$. More precisely,
		\begin{equation}
		\left\{\begin{array}{ll}
		d(\tau)=2n \text{ for all } \tau\in P_{n,1} & \text{if $n$ is odd} \\
		d(\tau)\in\{2n,n\} \text{ for all } \tau\in P_{n,1} & \text{if $n\equiv 0(\text{mod } 4)$} \\
		d(\tau)\in\{2n,n,\frac{n}{2}\} \text{ for all } \tau\in P_{n,1} & \text{if $n\equiv 2(\text{mod } 4)$} \\
		\end{array}\right. .
		\end{equation}
	\end{lemma}
	
	\begin{proof}
		
		First, note that $d(\tau)$ must be a divisor of $2n$. Now, assume $\tau\in P_{n,1}$, we use the bijection described earlier. Let $S_{\tau}=\{a_1, a_2, \dots, a_{n-2}\}$. Note that
		
		\begin{equation}
		S_{\sigma_{2n}(\tau)}=\{a_1+1,a_2+1,\dots, a_{n-2}+1\}(\text{mod } 2n).
		\end{equation}
		
		Since $\sigma_{2n}^{d(\tau)}(\tau)=\tau$. Hence, $S_{\tau}$ and $S_{\sigma_{2n}(\tau)}$ represent the same residue classes modulo $2n$. Thus $\sum_{i=1}^{n-2} a_i\equiv \sum_{i=1}^{n-2}(a_i+d(\tau))(\text{mod } 2n)$. Cancel the term $\sum_{i=1}^{n-2} a_i$ on both sides, and we get $2n|(n-2)d(\tau)$.
		
		\begin{itemize}
			\item If $n$ is odd, then $n-2$ is also odd. Thus, $\gcd(2n,n-2)=1\Rightarrow 2n|d(\tau)\Rightarrow d(\tau)=2n$.
			\item If $n$ is even, then $n|\frac{n-2}{2}\cdot d(\tau)$.
			\begin{itemize}
				\item If $4|n$, then $\frac{n-2}{2}$ is odd, and $\gcd(n,\frac{n-2}{2})=1$. Thus, $n|d(\tau)\Rightarrow d(\tau)\in \{2n,n\}$.
				\item If $n\equiv 2(\text{mod } 4)$, then $\frac{n}{2}|\frac{n-2}{4}\cdot d(\tau)$ and $\gcd(\frac{n}{2},\frac{n-2}{4})=1$. Thus, $\frac{n}{2}|d(\tau)\Rightarrow d(\tau)\in\{2n,n,\frac{n}{2}\}$.
			\end{itemize}
		\end{itemize}
		
	\end{proof}
	
	Now, let
	
	\begin{equation}
	A_j=\{\tau\in P_{n,1}|\sigma_{2n}^j(\tau)=\tau\}
	\end{equation}
	
	be the subset of one-crossing matchings fixed by $\sigma_{2n}^j$. Since $\sigma_{2n}^{2n}=1$, we have $|A_{2n}|=|P_{n,1}|=\binom{2n}{n-2}$. Note that for $\tau\in P_{n,1}$ and each fixed $1\leqslant j\leqslant 2n$, $\sigma_{2n}^j(\tau)=\tau$ if and only if $d(\tau)|j$. So, $A_j$ can also be written as
	
	\begin{equation}
	A_j=\{\tau\in P_{n,1}|\sigma_{2n}^j(\tau)=\tau\}=\{\tau\in P_{n,1}| d(\tau)\text{ is a divisor of }j\}.
	\end{equation}
	
	To check whether $f_{n,1}(p)$ satisfies the conditions of cyclic sieving, it is sufficient to check whether $f_{n,1}(\xi_{2n}^j)$ agrees with $|A_j|$ at each $j$. The values of $|A_j|$ are evaluated in the following lemma.
	
	\begin{lemma}\thlabel{size} 
		$|A_{2n}|=\binom{2n}{n-2}$, and $|A_j|=0$ for $j\notin \{\frac{n}{2}, n, 2n\}$. The values of $A_n$ and $A_{\frac{n}{2}}$ depend on the value of $n(\text{mod } 4)$:
		\begin{equation}
		|A_n|=
		\left\{\begin{array}{ll}
		0 & \text{ if $n$ is odd}\\
		\binom{n}{\frac{n-2}{2}} & \text{ if $n$ is even}\
		\end{array}\right. , 
		\end{equation}
		
		\begin{equation}
		|A_{\frac{n}{2}}|=
		\left\{\begin{array}{ll}
		0 & \text{ if $n$ is odd or $4|n$}\\
		\binom{\frac{n}{2}}{\frac{n-2}{4}} & \text{ if $n\equiv 2 \pmod 4$ }\\
		\end{array}\right. .
		\end{equation}
	\end{lemma}
	
	\begin{proof}
		
		From the definition, we have $|A_{2n}|=|P_{n,1}|=\binom{2n}{n-2}$. Also, $|A_j|=0$ for $j\notin \{\frac{n}{2}, n, 2n\}$ follows immediately from \thref{dtau}. Now, we calculate $|A_n|$ and $|A_{\frac{n}{2}}|$. Again, according to \thref{dtau}, we have $|A_n|=0$ when $n$ is odd and $|A_{\frac{n}{2}}|=0$ when $n$ is odd or $4|n$. 
		
		Now, when $n$ is even, we divide the numbers $\{1,2,\dots, 2n\}$ into $n$ pairs as the following:
		
		\begin{align*}
		& \beta_1=\{1, 1+n\},\\
		& \beta_2=\{2, 2+n\},\\
		& \dots \\
		& \beta_n=\{n, 2n\}.
		\end{align*}
		
		For $\tau\in P_{n,1}$, we use the bijective relation again. Let $S_{\tau}=\{a_1, a_2, \dots, a_{n-2}\}$, then $S_{\sigma_{2n}^n(\tau)}=\{a_1+n, a_2+n, \dots, a_{n-2}+n\}$. From the bijection, we know $\tau\in A_n$ if and only if $S_{\tau}=S_{\sigma_{2n}^n(\tau)} \pmod{2n}$, which means for each $1\leqslant i\leqslant n$ we have $i\in S_{\tau}\Leftrightarrow i+n\in S_{\tau}$. Hence,
		
		\begin{equation}
		|A_n|=\# \text{ ways to pick $\frac{n-2}{2}$ pairs among $\beta_j$'s}=\binom{n}{\frac{n-2}{2}}.
		\end{equation}
		
		Now, when $n\equiv 2 \pmod 4$, we divide the numbers $\{1,2,\dots, 2n\}$ into $\frac{n}{2}$ parts as the following:
		
		\begin{align*}
		& \gamma_1=\{1, 1+\frac{n}{2}, 1+n, 1+\frac{3n}{2}\},\\
		& \gamma_2=\{2, 2+\frac{n}{2}, 2+n, 2+\frac{3n}{2}\},\\
		& \dots \\
		& \gamma_{\frac{n}{2}}=\{\frac{n}{2}, n, \frac{3n}{2}, 2n\}.
		\end{align*}
		
		We now have $S_{\sigma_{2n}^{\frac{n}{2}}(\tau)}=\{a_1+\frac{n}{2}, a_2+\frac{n}{2}, \dots, a_{n-2}+\frac{n}{2}\}$. From the bijective relation, we know $\tau\in A_{\frac{n}{2}}$ if and only if $S_{\tau}=S_{\sigma_{2n}^{\frac{n}{2}}(\tau)} \pmod{2n}$, which means the following four statements are equivalent for each $1\leqslant i\leqslant n$: 1) $i\in S_{\tau}$, 2) $i+\frac{n}{2} \in S_{\tau}$, 3)$i+n\in S_{\tau}$, and 4) $i+\frac{3n}{2} \in S_{\tau}$. Hence, 
		
		\begin{equation}
		|A_{\frac{n}{2}}|=\# \text{ ways to pick $\frac{n-2}{4}$ parts among $\gamma_j$'s}=\binom{\frac{n}{2}}{\frac{n-2}{4}}.
		\end{equation}
	\end{proof}
	
	With these results, we can now conclude the main result for the one crossing matchings. 
	
	\begin{thm}[CSP for One Crossing Matchings]\thlabel{main1}
		Let $\xi_{2n}=e^{\frac{\pi i}{n}}$, which is a primitive root of the equation $x^{2n}-1=0$. Let $f_{n,1}=\binomial{2n}{n-2}_q$. Then $f_{n,1}(\xi_{2n}^j)=|\{\tau\in P_{n,1}|\sigma_{2n}^j(\tau)=\tau\}|$. In other words, $(P_{n,1},f_{n,1}(q),C_{2n})$ exhibits the cyclic sieving phenomenon.
	\end{thm}
	
	\begin{proof}
		
		Now, we can use the results from previous lemmas to prove $f_{n,1}(\zeta_{2n}^j)=|A_j|$. First, note that $f_{n,1}(1)=\binom{2n}{n-2}=|A_{2n}|$ is clear. Then, we prove 
		
		\begin{equation}
		f_{n}(-1)=
		\left\{\begin{array}{ll}
		0=|A_{n}| & \text{ if $n$ is odd}\\
		\binom{n}{\frac{n-2}{2}}=|A_{n}| & \text{ if $n$ is even}\\
		\end{array}\right. ,
		\end{equation}
		
		using the following facts.
		
		\begin{fact}\thlabel{fact1}
			When $m$ is odd,
			\begin{equation}
			[m]_q(-1)=1+(-1)^1+(-1)^2+\dots+(-1)^{m-1}=1.
			\end{equation}
		\end{fact}
		
		\begin{fact}\thlabel{fact2}
			When $m$ is even,
			\begin{equation}
			[m]_q=(1+q)(1+q^2+q^4+\dots+q^{m-2}).
			\end{equation}
			Consequently, 
			\begin{equation}
			[m]_q(-1)=0,
			\end{equation}
			and for any two positive even integers $m_1$ and $m_2$, we have
			\begin{equation}
			\frac{[m_1]_q}{[m_2]_q}(-1)=\frac{1+(-1)^2+(-1)^4+\dots+(-1)^{m_1-2}}{1+(-1)^2+(-1)^4+\dots+(-1)^{m_2-2}}=\frac{m_1/2}{m_2/2}.
			\end{equation}
		\end{fact}
		
		Now, when $n$ is even, we have 
		
		\begin{align*}
		f_{n,1}(-1)=\binomial{2n}{n-2}_{q=-1} & = \frac{[2n]_{q=-1}[2n-1]_{q=-1}\cdots [n+3]_{q=-1}}{[n-2]_{q=-1}[n-3]_{q=-1}\cdots [1]_{q=-1}}\\
		& = \frac{[2n]_{q=-1}[2n-2]_{q=-1}\cdots [n+4]_{q=-1}}{[n-2]_{q=-1}[n-4]_{q=-1}\cdots [2]_{q=-1}}, \text{ (We have used \thref{fact1} here.)}\\
		& = \frac{[2n]_{q=-1}}{[n-2]_{q=-1}}\frac{[2n-2]_{q=-1}}{[n-4]_{q=-1}}\cdots \frac{[n+4]_{q=-1}}{ [2]_{q=-1}}\\
		& = \frac{n(n-1)\cdots \frac{n+4}{2}}{\frac{n-2}{2}\frac{n-4}{2}\cdots 1}, \text{ (We have used \thref{fact2} here.)}\\
		& = \binom{n}{\frac{n-2}{2}}=|A_n|.
		\end{align*}
		
		But when $n$ is odd, we have 
		
		\begin{align*}
		f_{n,1}(-1)=\binom{2n}{n-2}_{q=-1} & = \frac{[2n]_{q=-1}[2n-1]_{q=-1}\cdots [n+3]_{q=-1}}{[n-2]_{q=-1}[n-3]_{q=-1}\cdots [1]_{q=-1}}\\
		& = \frac{[2n]_{q=-1}[2n-2]_{q=-1}\cdots [n+3]_{q=-1}}{[n-3]_{q=-1}\cdots [2]_{q=-1}}, \text{ (We have used \thref{fact1} here.)}\\
		& = [2n]_{q=-1}\frac{[2n-2]_{q=-1}}{[n-3]_{q=-1}}\frac{[2n-4]_{q=-1}}{[n-5]_{q=-1}}\cdots \frac{[n+3]_{q=-1}}{ [2]_{q=-1}}\\
		& = 0\cdot \frac{(n-1)(n-2)\cdots \frac{n+3}{2}}{\frac{n-3}{2}\frac{n-5}{2}\cdots 1}, \text{ (We have used \thref{fact2} here.)}\\
		& = 0=|A_n|.
		\end{align*}
		
		Using the same idea, we can prove $f_{n,1}(\zeta_{2n}^j)=|A_j|$. Thus, the proof is complete and $(P_{n,1},f_{n,1}(q),C_{2n})$ exhibits the cyclic sieving phenomenon.
	\end{proof}
	
	\subsection{Two-crossing Matchings}
	
	\begin{lemma}\thlabel{bijection2}
		$|P_{n,2}|=\frac{n+3}{2}\binom{2n}{n-3}$.
	\end{lemma}
	
	\begin{proof}
		The $\NCC$ algorithm is again useful here, but more classifications are needed before we apply the algorithm. The reason is that different from the one-crossing matchings, the two crossings can be separated into two parts by some edges. There are $n-2$ types of two-crossings matchings in total, according to the number of edges separating the two crosses. 
		
		There are two types of two-crossing matchings such that no edge separates the two crossings. The first type, denoted by $T_3$, contains the two crossing matchings with the following partial matching. 
		
		\begin{center}
			\begin{tikzpicture}
			\def \n {6}
			\def \radius {1cm}
			
			\foreach \s in {\n,...,1}
			{
				\draw[-, >=latex] ({360/\n * (\s - 1)}:\radius) 
				arc ({360/\n * (\s -1)}:{360/\n * (\s)}:\radius);
			}
			
			\foreach \from/\to in {1/5, 2/4, 3/6}
			\draw ({360/\n * (\from - 1)}:\radius) -- ({360/\n * (\to - 1)}:\radius);
			\node[text width=6cm, anchor=west, right] at (-0.3,-1.5)
			{$T_3$};
			\end{tikzpicture}
		\end{center}
		
		After we choose an $(n-3)$-element subset $S\subseteq\{1,2,\cdots,2n\}$ and apply the algorithm, $\NCC(n,S)$ returns a noncrossing partial matchings with $6$ points left unmatched. This gives us three ways to construct a $T_3$-type of two crossing matching. Thus, there is a $3$ to $1$ corresponding relation between this type of two-crossing matchings and the $(n-3)$-element subsets $S\subseteq\{1,2,\cdots,2n\}$, which implies the number of elements in this type is $|T_3|=3\binom{2n}{n-3}$.
		
		Here is an explicit example with $n = 7$ and $S = \{1,2,5,9\}$. Then $\NCC(7,\{1,2,5,9\})$ gives us a noncrossing partial matching $(1,4)(2,3)(5,6)(9,10)$. 
		
		
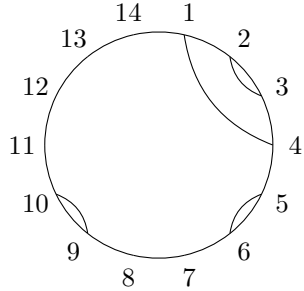
\begin{center}
			\begin{tikzpicture}
			
			\def \n {14}
			\def \radius {1.5cm}
			
			\foreach \s in {\n,...,1}
			{
				\node at ({-360/\n * (\s-4)}:1.2*\radius) {$\s$} ;
				\draw[-, >=latex] ({360/\n * (\s - 1)}:\radius) 
				arc ({360/\n * (\s -1)}:{360/\n * (\s)}:\radius);
			}
			
			\foreach \from/\to in {2/3, 1/4, 9/10, 13/14}
			\draw ({360/\n * (\from - 1)}:\radius) to [bend left=30] ({360/\n * (\to - 1)}:\radius);
			
			\end{tikzpicture}
		\end{center}
		
		There are exactly six points $7,8,11,12,13,14$ left, which give us three ways to get a two-crossing matching in the type of $T_3$, which are $(7,12)(8,14)(11,13)$, $(8,13)(7,11)(12,14)$, and $(11,14)(8,12)(7,13)$ (shown in the following graphs). Thus, we can conclude that $|T_3|=3\binom{2n}{n-3}$.
		
		
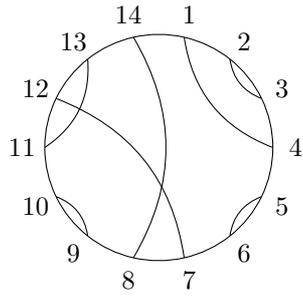
\begin{center}
			\begin{tikzpicture}
			
			\def \n {14}
			\def \radius {1.5cm}
			
			\foreach \s in {\n,...,1}
			{
				\node at ({-360/\n * (\s-4)}:1.2*\radius) {$\s$} ;
				\draw[-, >=latex] ({360/\n * (\s - 1)}:\radius) 
				arc ({360/\n * (\s -1)}:{360/\n * (\s)}:\radius);
			}
			
			\foreach \from/\to in {2/3, 1/4, 9/10, 13/14, 7/12, 5/11, 6/8}
			\draw ({360/\n * (\from - 1)}:\radius) to [bend left=30] ({360/\n * (\to - 1)}:\radius);
			
			\end{tikzpicture}
		\end{center}
		
		
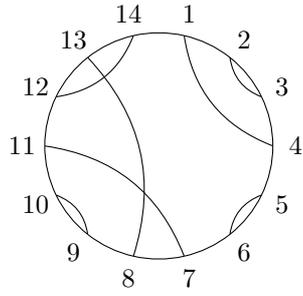
\begin{center}
			\begin{tikzpicture}
			
			\def \n {14}
			\def \radius {1.5cm}
			
			\foreach \s in {\n,...,1}
			{
				\node at ({-360/\n * (\s-4)}:1.2*\radius) {$\s$} ;
				\draw[-, >=latex] ({360/\n * (\s - 1)}:\radius) 
				arc ({360/\n * (\s -1)}:{360/\n * (\s)}:\radius);
			}
			
			\foreach \from/\to in {2/3, 1/4, 9/10, 13/14, 8/12, 6/11, 5/7}
			\draw ({360/\n * (\from - 1)}:\radius) to [bend left=30] ({360/\n * (\to - 1)}:\radius);
			
			\end{tikzpicture}
		\end{center}
		
		
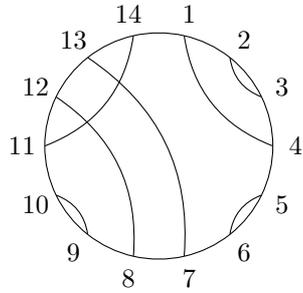
\begin{center}
			\begin{tikzpicture}
			
			\def \n {14}
			\def \radius {1.5cm}
			
			\foreach \s in {\n,...,1}
			{
				\node at ({-360/\n * (\s-4)}:1.2*\radius) {$\s$} ;
				\draw[-, >=latex] ({360/\n * (\s - 1)}:\radius) 
				arc ({360/\n * (\s -1)}:{360/\n * (\s)}:\radius);
			}
			
			\foreach \from/\to in {2/3, 1/4, 9/10, 13/14, 5/8, 7/11, 6/12}
			\draw ({360/\n * (\from - 1)}:\radius) to [bend left=30] ({360/\n * (\to - 1)}:\radius);
			
			\end{tikzpicture}
		\end{center}
		
		Another type of two-crossing matchings with no edge separates the two different crossings, denoted by $T_4$, is the set of the two crossing matchings contain the following partial matching. Similarly, we have $|T_4|=4\binom{2n}{n-4}$.
		
		\begin{center}
			\begin{tikzpicture}
			\def \n {8}
			\def \radius {1cm}
			
			\foreach \s in {\n,...,1}
			{
				\draw[-, >=latex] ({360/\n * (\s - 1)}:\radius) 
				arc ({360/\n * (\s -1)}:{360/\n * (\s)}:\radius);
			}
			
			\foreach \from/\to in {1/7, 6/8, 2/4, 3/5}
			\draw ({360/\n * (\from - 1)}:\radius) -- ({360/\n * (\to - 1)}:\radius);
			\node[text width=6cm, anchor=west, right] at (-0.3,-1.5)
			{$T_4$};
			\end{tikzpicture}
		\end{center}
		
		Here is an example with $n = 7$ and $S = \{1,2,5\}$. After we chose the $7-4=3$ points, $\NCC(n,\{1,2,5\})$ gives us a noncrossing partial matching $(1,4)(2,3)(5,6)$. There are eight points $7,8,11,12,13,14$ left, which give us exactly four ways to get a two-crossing matching in the type of $T_4$, which are $(7,9)(8,10)(11,13)(12,14)$, $(8,10)(9,11)(12,14)(7,13)$, $(9,11)(10,12)(7,13)(8,14)$, and $(10,12)(11,13)(8,14)(7,9)$.
		
		
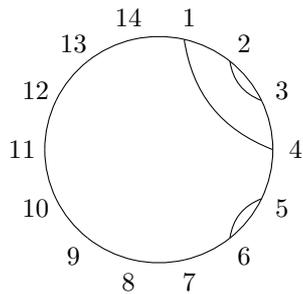
\begin{center}
			\begin{tikzpicture}
			
			\def \n {14}
			\def \radius {1.5cm}
			
			\foreach \s in {\n,...,1}
			{
				\node at ({-360/\n * (\s-4)}:1.2*\radius) {$\s$} ;
				\draw[-, >=latex] ({360/\n * (\s - 1)}:\radius) 
				arc ({360/\n * (\s -1)}:{360/\n * (\s)}:\radius);
			}
			
			\foreach \from/\to in {2/3, 1/4, 13/14}
			\draw ({360/\n * (\from - 1)}:\radius) to [bend left=30] ({360/\n * (\to - 1)}:\radius);
			
			\end{tikzpicture}
		\end{center}
		
		There can be some edges separate the two different crossings. The matchings with exactly $k-4$ edges separating the two crossings form a type, denoted by $T_k$, where $5\leqslant k\leqslant n$. Moreover, there is a $k$ to $1$ corresponding relation between $T_k$ and the noncrossing partial matchings obtained by $\NCC(n,S)$ with $S\subseteq \{1,2,\dots,2n\}$ and $|S|=n-k$, which implies
		
		\begin{lemma}\thlabel{type}
			The number of elements in the type $T_k$ is 
			
			$$|T_k|=k\binom{2n}{n-k}, \text{ for } 3\leqslant k\leqslant n.$$
		\end{lemma}
		\begin{proof}[Proof of \thref{type}]
			We have proved it for $k=3,4$. Now, assume $k\geqslant 5$. The $T_k$ type has $k-4$ edges separate the two crossings (shown below). 
			
			
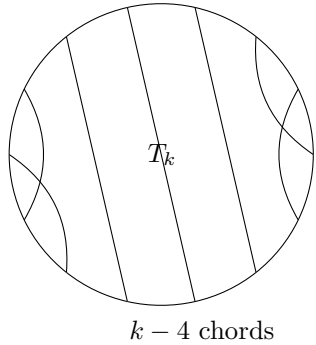
\begin{center}
				\begin{tikzpicture}
				
				\def \n {14}
				\def \radius {2cm}
				
				\node at ({-360/\n * (1-4)}:1.2*\radius) {} ;
				\draw[-, >=latex] ({360/\n * (1 - 1)}:\radius) 
				arc ({360/\n * (1 -1)}:{360/\n * (1)}:\radius);
				
				\node at ({-360/\n * (2-4)}:1.2*\radius) {} ;
				\draw[-, >=latex] ({360/\n * (2 - 1)}:\radius) 
				arc ({360/\n * (2 -1)}:{360/\n * (2)}:\radius);
				
				\node at ({-360/\n * (3-4)}:1.2*\radius) {} ;
				\draw[-, >=latex] ({360/\n * (3 - 1)}:\radius) 
				arc ({360/\n * (3 -1)}:{360/\n * (3)}:\radius);
				
				\node at ({-360/\n * (4-4)}:1.2*\radius) {} ;
				\draw[-, >=latex] ({360/\n * (4 - 1)}:\radius) 
				arc ({360/\n * (4 -1)}:{360/\n * (4)}:\radius);
				
				\node at ({-360/\n * (5-4)}:1.2*\radius) {} ;
				\draw[-, >=latex] ({360/\n * (5 - 1)}:\radius) 
				arc ({360/\n * (5 -1)}:{360/\n * (5)}:\radius);
				
				\node at ({-360/\n * (6-4)}:1.2*\radius) {} ;
				\draw[-, >=latex] ({360/\n * (6 - 1)}:\radius) 
				arc ({360/\n * (6 -1)}:{360/\n * (6)}:\radius);
				
				\node at ({-360/\n * (7-4)}:1.2*\radius) {$k-4$ chords} ;
				\draw[-, >=latex] ({360/\n * (7 - 1)}:\radius) 
				arc ({360/\n * (7 -1)}:{360/\n * (7)}:\radius);
				
				\node at ({-360/\n * (8-4)}:1.2*\radius) {} ;
				\draw[-, >=latex] ({360/\n * (8 - 1)}:\radius) 
				arc ({360/\n * (8 -1)}:{360/\n * (8)}:\radius);
				
				\node at ({-360/\n * (9-4)}:1.2*\radius) {} ;
				\draw[-, >=latex] ({360/\n * (9 - 1)}:\radius) 
				arc ({360/\n * (9 -1)}:{360/\n * (9)}:\radius);
				
				\node at ({-360/\n * (10-4)}:1.2*\radius) {} ;
				\draw[-, >=latex] ({360/\n * (10 - 1)}:\radius) 
				arc ({360/\n * (10 -1)}:{360/\n * (10)}:\radius);
				
				\node at ({-360/\n * (11-4)}:1.2*\radius) {} ;
				\draw[-, >=latex] ({360/\n * (11 - 1)}:\radius) 
				arc ({360/\n * (11 -1)}:{360/\n * (11)}:\radius);
				
				\node at ({-360/\n * (12-4)}:1.2*\radius) {} ;
				\draw[-, >=latex] ({360/\n * (12 - 1)}:\radius) 
				arc ({360/\n * (12 -1)}:{360/\n * (12)}:\radius);
				
				\node at ({-360/\n * (13-4)}:1.2*\radius) {} ;
				\draw[-, >=latex] ({360/\n * (13 - 1)}:\radius) 
				arc ({360/\n * (13 -1)}:{360/\n * (13)}:\radius);
				
				\node at ({-360/\n * (14-4)}:1.2*\radius) {} ;
				\draw[-, >=latex] ({360/\n * (14 - 1)}:\radius) 
				arc ({360/\n * (14 -1)}:{360/\n * (14)}:\radius);
				
				\foreach \from/\to in {14/2,1/3,7/9,8/10}
				\draw ({360/\n * (\from - 1)}:\radius) to [bend left=30] ({360/\n * (\to - 1)}:\radius);
				
				\foreach \from/\to in {4/13,5/12,6/11}
				\draw ({360/\n * (\from - 1)}:\radius) -- ({360/\n * (\to - 1)}:\radius);
				\node[text width=6cm, anchor=west, right] at (-0.3,0)
				{$T_k$};
				\end{tikzpicture}
			\end{center}
			
			So, if we choose any subet $S\subseteq \{1,2,\dots,2n\}$ with $|S|=n-k$, by $\NCC(n,S)$ we can construct a noncrossing partial matching with $2k$ vertices unmatched. Among these $2k$ vertices, there are exactly $k$ ways to get a two crossing mathching in the type of $T_k$.
		\end{proof}
		
		For example, when $k=5$, we get the type $T_5$ with exactly one edges separating the two crossings, which is shown as the following.
		
		
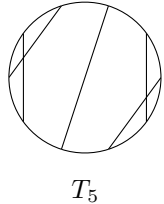
\begin{center}
			\begin{tikzpicture}
			\def \n {10}
			\def \radius {1cm}
			
			\foreach \s in {\n,...,1}
			{
				\draw[-, >=latex] ({360/\n * (\s - 1)}:\radius) 
				arc ({360/\n * (\s -1)}:{360/\n * (\s)}:\radius);
			}
			
			\foreach \from/\to in {1/9, 2/10, 4/6, 5/7, 3/8}
			\draw ({360/\n * (\from - 1)}:\radius) -- ({360/\n * (\to - 1)}:\radius);
			\node[text width=6cm, anchor=west, right] at (-0.3,-1.5)
			{$T_5$};
			\end{tikzpicture}
		\end{center}
		
		Here is an example in $T_5$. Again, for $n = 7$ and $S = \{1,2\}$. After we chose the $7-5=2$ points, $\NCC(7,\{1,2\})$ gives us a non-crossing sub-matching $(1,4)(2,3)$. There are ten points (i.e. $5,\cdots, 14$) left, which give us exactly five ways to get a two-crossing matching in the type of $T_5$. Thus, we can conclude that $|T_5|=5\binom{2n}{n-5}$.
		
		
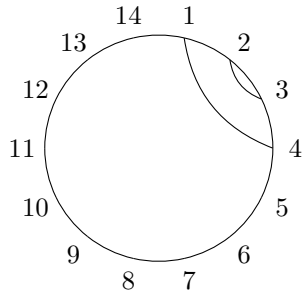
\begin{center}
			\begin{tikzpicture}
			
			\def \n {14}
			\def \radius {1.5cm}
			
			\foreach \s in {\n,...,1}
			{
				\node at ({-360/\n * (\s-4)}:1.2*\radius) {$\s$} ;
				\draw[-, >=latex] ({360/\n * (\s - 1)}:\radius) 
				arc ({360/\n * (\s -1)}:{360/\n * (\s)}:\radius);
			}
			
			\foreach \from/\to in {2/3, 1/4}
			\draw ({360/\n * (\from - 1)}:\radius) to [bend left=30] ({360/\n * (\to - 1)}:\radius);
			
			\end{tikzpicture}
		\end{center}
		
		Now, by using the combinatorial equalities $\sum_{j=0}^s \binom{s}{j}=2^s$ and $\binom{s}{j}=\binom{s}{s-j}$, we can prove the lemma. Note that $T_3,\cdots, T_n$ are disjoint cover of $P_{n,2}$.
		
		\begin{align*}
		|P_{n,2}| =\sum_{k=3}^n|T_k|& = \sum_{k=3}^n k\binom{2n}{n-k} =\sum_{i=0}^{n-3} (n-i)\binom{2n}{i}\\
		& =\sum_{i=0}^{n-1} (n-i)\binom{2n}{i}-2\binom{2n}{n-2}-\binom{2n}{n-1}\\
		& =n\sum_{i=0}^{n-1}\binom{2n}{i}-\sum_{i=1}^{n-1}i\binom{2n}{i}-2\binom{2n}{n-2}-\binom{2n}{n-1}\\
		& =n\cdot \frac{2^{2n}-\binom{2n}{n}}{2}-\sum_{i=1}^{n-1}2n\binom{2n-1}{i-1}-2\binom{2n}{n-2}-\binom{2n}{n-1}\\
		& =n(2^{2n-1}-\frac{1}{2}\binom{2n}{n})-2n\sum_{i=0}^{n-2}\binom{2n-1}{i}-2\binom{2n}{n-2}-\binom{2n}{n-1}\\
		& =n(2^{2n-1}-\frac{1}{2}\binom{2n}{n})-2n(\frac{2^{2n-1}}{2}-\binom{2n-1}{n-1})-2\binom{2n}{n-2}-\binom{2n}{n-1}\\
		& =n[2\binom{2n-1}{n-1}-\frac{1}{2}\binom{2n}{n}]-2\binom{2n}{n-2}-\binom{2n}{n-1}\\
		& =\frac{n+1}{2}\binom{2n}{n-1}-2\binom{2n}{n-2}-\binom{2n}{n-1}\\
		& =\frac{n+3}{2}\binom{2n}{n-3}
		\end{align*}
	\end{proof}
	
	Now, let
	
	\begin{equation}
	B_j=\{\tau\in P_{n,2}|\sigma_{2n}^j(\tau)=\tau\}
	\end{equation}
	
	be the subset of two-crossing matchings fixed by $\sigma_{2n}^j$.
	
	\begin{lemma}\thlabel{size2} 
		$|B_{2n}|=\frac{n+3}{2}\binom{2n}{n-3}$, and $|B_j|=0$ for $j\notin \{n, 2n\}$. Moreover,
		\begin{equation}
		|B_n|=
		\left\{\begin{array}{ll}
		\frac{n-1}{2}\binom{n}{\frac{n-1}{2}} & \text{ if $n$ is odd}\\
		\frac{n-2}{2}\binom{n}{\frac{n-2}{2}} & \text{ if $n$ is even}\
		\end{array}\right. . 
		\end{equation}
	\end{lemma}
	\begin{proof}
		It is easy to see that to get the same two-crossing matching by rotation, the only two possibilities are rotating one circle or one semi-circle, since each crossing must be sent to the position of the other crossing or itself. Thus, we have $|B_j|= 0$ for $j\notin\{n,2n\}$. Also, from the definition, we have $B_{2n}= P_{n,2}$.
		
		Now, we only need to check $|B_n|$. We only prove the case in which $n$ is odd in detail, the the same idea can be used to prove the case $n$ is even.
		
		Assume $n$ is odd. Note that the total number of edges is $n$. Hence, $B_n\cap T_k =\emptyset$ if $k$ is even, which implies that
		
		$$|B_n|=|B_n\cap T_3|+|B_n\cap T_5|+\cdots +|B_n\cap T_n|.$$
		
		Note that there is a $k$ to $1$ corresponding relation between $T_k$ and the $(n-k)$-element subsets of $\{1,2,\dots,2n\}$. Moreover, for any $\tau\in T_k$, we have $\tau\in B_n$ if any only if the corresponding $(n-k)$-element subset $S$ satisfies $i\in S \Leftrightarrow i + n \pmod{2n}\in S$. By pairing the $2n$ points as $(1, n + 1),\dots, (n,2n)$, it is equivalent to choose $\frac{n-k}{2}$ pairs of points from the $n$ pairs, which implies
		
		$$|B_n\cap T_k|=k\binom{n}{\frac{n-k}{2}}.$$
		
		Then, by the similar calculation as we used to calculate $|P_{n,2}|$, we have
		
		$$|B_n|=3\binom{n}{\frac{n-3}{2}}+5\binom{n}{\frac{n-5}{2}}+\cdots +n\binom{n}{0}=\frac{n-1}{2}\binom{n}{\frac{n-1}{2}}.$$
	\end{proof}
	
	With \thref{size2} proved, the only thing remained to check is the q-analog $f_{n,2}(\zeta_{2n}^j)$ agrees with $|B_j|$ at each $1\leqslant j\leqslant 2n$. 
	
	\begin{thm}[CSP for Two Crossing Matchings]\thlabel{main2}
		Let $\xi_{2n}=e^{\frac{\pi i}{n}}$, which is a primitive root of the equation $x^{2n}-1=0$. Let $f_{n,2}=\frac{[n+3]_q}{[2]_q}\binomial{2n}{n-3}_q$. Then $f_{n,2}(\xi_{2n}^j)=|\{\tau\in P_{n,2}|\sigma_{2n}^j(\tau)=\tau\}|$. In other words, $(P_{n,2},f_{n,2}(q),C_{2n})$ exhibits the cyclic sieving phenomenon.
	\end{thm}
	
	\begin{proof}
		As an example, we prove $f_{n,2}(-1)=|B_n|$ when $n$ is odd. Other cases can be proven by the same idea. Applying \thref{fact1} and \thref{fact2}, we have
		
		\begin{align*}
		f_{n,2}(-1)=(\frac{[n+3]_q}{[2]_q}\binomial{2n}{n-3}_q)_{q=-1} & = \frac{[n+3]_{q=-1}}{[2]_{q=-1}}\frac{[2n]_{q=-1}[2n-1]_{q=-1}\cdots [n+5]_{q=-1}[n+4]_{q=-1}}{[n-3]_{q=-1}[n-2]_{q=-1}\cdots [2]_{q=-1}[1]_{q=-1}}\\
		& = \frac{[n+3]_{q=-1}}{[2]_{q=-1}}\frac{[2n]_{q=-1}[2n-2]_{q=-1}\cdots [n+5]_{q=-1}}{[n-3]_{q=-1}[n-5]_{q=-1}\cdots [2]_{q=-1}}, \text{ (We have used \thref{fact1} here.)}\\
		& = \frac{[2n]_{q=-1}}{[n-3]_{q=-1}}\frac{[2n-2]_{q=-1}}{[n-5]_{q=-1}}\cdots \frac{[n+5]_{q=-1}}{ [2]_{q=-1}}\frac{[n+3]_{q=-1}}{[2]_{q=-1}}\\
		& = \frac{n(n-1)\cdots \frac{n+5}{2}}{\frac{n-3}{2}\frac{n-5}{2}\cdots 1}\frac{n+3}{2}, \text{ (We have used \thref{fact2} here.)}\\
		& = \frac{n-1}{2}\frac{n(n-1)\cdots \frac{n+5}{2}\frac{n+3}{2}}{\frac{n-1}{2}\frac{n-3}{2}\cdots 2\cdot 1}\\
		& = \frac{n-1}{2}\binom{n}{\frac{n-1}{2}}=|B_n|.
		\end{align*}
	\end{proof}
	
	So we can conclude that $(P_{n,2},f_{n,2}(q), C_{2n})$ exhibits the Cyclic Sieving Phenomenon. 
	
	\subsection{Three-crossing Matchings}
	
	The ideals of proving the Cyclic Sieving Phenomenon of three-crossing matchings are completed covered by the methods we used to prove the one-crossing and two-crossing cases. 
	
	\begin{thm}[CSP for Three Crossing Matchings]\thlabel{main3}
		Let $\xi_{2n}=e^{\frac{\pi i}{n}}$, which is a primitive root of the equation $x^{2n}-1=0$. Let $f_{n,3}(q)=\frac{1}{[3]_q} \binomial{n+5}{2}_q \binomial{2n}{n-4}_q+\binomial{2n}{n-3}_q$. Then $f_{n,3}(\xi_{2n}^j)=|\{\tau\in P_{n,3}|\sigma_{2n}^j(\tau)=\tau\}|$. In other words, $(P_{n,3},f_{n,3}(q),C_{2n})$ exhibits the cyclic sieving phenomenon.
	\end{thm}
	
	\begin{proof}
		Using the same technique, we can check the following:
		
		$$f_{n,3}(1)=|P_{n,3}|=\frac{1}{3}\binom{n+5}{2}\binom{2n}{n-4}+\binom{2n}{n-3},$$
		
		\begin{equation*}
		f_{n,3}(-1)=|\{\tau\in P_{n,3}|\sigma_{2n}^n(\tau)=\tau\}|=
		\left\{
		\begin{array}{ll}
		\binom{n}{\frac{n-3}{2}} & \text{ if $n$ is odd} \\
		\frac{n+4}{2}\binom{n}{\frac{n-4}{2}} & \text{ if $n$ is even} 
		\end{array}\right.,
		\end{equation*}
		
		and
		
		$$f_{n,3}(\xi_{2n}^j)=|\{\tau\in P_{n,3}|\sigma_{2n}^j(\tau)=\tau\}|=0, \text{ for } j\neq 0,\frac{2n}{3},n,\frac{4n}{3}.$$
		
		Note that different from one crossing matchings and two crossing matchings, some of the three crossing matchings can be mapped to itself by rotating one-third of the circle (i.e $\{\tau\in P_{n,3}|\sigma_{2n}^{\frac{2n}{3}}(\tau)=\tau\}\neq \emptyset$ if $3|n$). So we still need to check that 
		
		\begin{equation}\label{one-third}
		f_{n,3}(u)=f_{n,3}(u^2)=|\{\tau\in P_{n,3}|\sigma_{2n}^{\frac{2n}{3}}(\tau)=\tau\}|,
		\end{equation}
		
		where $u=e^{\frac{2\pi i}{3}}$. 
		
		For $3|n$, we apply $\NCC$ algorithm. Let $F=\{\tau\in P_{n,3}|\sigma_{2n}^{\frac{2n}{3}}(\tau)=\tau\}$, the set of three crossing matchings that are invariant under the map $\sigma_{2n}^{\frac{2n}{3}}$. Again, we divide $F$ into different types according to the number of edges separating the three crossing. The first type $R_1$ is the three crossing matchings containing the following submatching:
		
		
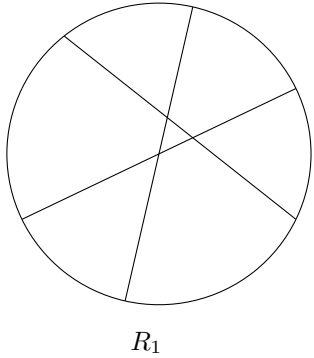
\begin{center}
			\begin{tikzpicture}
			
			\def \n {14}
			\def \radius {2cm}
			
			\node at ({-360/\n * (1-4)}:1.2*\radius) {} ;
			\draw[-, >=latex] ({360/\n * (1 - 1)}:\radius) 
			arc ({360/\n * (1 -1)}:{360/\n * (1)}:\radius);
			
			\node at ({-360/\n * (2-4)}:1.2*\radius) {} ;
			\draw[-, >=latex] ({360/\n * (2 - 1)}:\radius) 
			arc ({360/\n * (2 -1)}:{360/\n * (2)}:\radius);
			
			\node at ({-360/\n * (3-4)}:1.2*\radius) {} ;
			\draw[-, >=latex] ({360/\n * (3 - 1)}:\radius) 
			arc ({360/\n * (3 -1)}:{360/\n * (3)}:\radius);
			
			\node at ({-360/\n * (4-4)}:1.2*\radius) {} ;
			\draw[-, >=latex] ({360/\n * (4 - 1)}:\radius) 
			arc ({360/\n * (4 -1)}:{360/\n * (4)}:\radius);
			
			\node at ({-360/\n * (5-4)}:1.2*\radius) {} ;
			\draw[-, >=latex] ({360/\n * (5 - 1)}:\radius) 
			arc ({360/\n * (5 -1)}:{360/\n * (5)}:\radius);
			
			\node at ({-360/\n * (6-4)}:1.2*\radius) {} ;
			\draw[-, >=latex] ({360/\n * (6 - 1)}:\radius) 
			arc ({360/\n * (6 -1)}:{360/\n * (6)}:\radius);
			
			\node at ({-360/\n * (7-4)}:1.2*\radius) {} ;
			\draw[-, >=latex] ({360/\n * (7 - 1)}:\radius) 
			arc ({360/\n * (7 -1)}:{360/\n * (7)}:\radius);
			
			\node at ({-360/\n * (8-4)}:1.2*\radius) {} ;
			\draw[-, >=latex] ({360/\n * (8 - 1)}:\radius) 
			arc ({360/\n * (8 -1)}:{360/\n * (8)}:\radius);
			
			\node at ({-360/\n * (9-4)}:1.2*\radius) {} ;
			\draw[-, >=latex] ({360/\n * (9 - 1)}:\radius) 
			arc ({360/\n * (9 -1)}:{360/\n * (9)}:\radius);
			
			\node at ({-360/\n * (10-4)}:1.2*\radius) {} ;
			\draw[-, >=latex] ({360/\n * (10 - 1)}:\radius) 
			arc ({360/\n * (10 -1)}:{360/\n * (10)}:\radius);
			
			\node at ({-360/\n * (11-4)}:1.2*\radius) {} ;
			\draw[-, >=latex] ({360/\n * (11 - 1)}:\radius) 
			arc ({360/\n * (11 -1)}:{360/\n * (11)}:\radius);
			
			\node at ({-360/\n * (12-4)}:1.2*\radius) {} ;
			\draw[-, >=latex] ({360/\n * (12 - 1)}:\radius) 
			arc ({360/\n * (12 -1)}:{360/\n * (12)}:\radius);
			
			\node at ({-360/\n * (13-4)}:1.2*\radius) {} ;
			\draw[-, >=latex] ({360/\n * (13 - 1)}:\radius) 
			arc ({360/\n * (13 -1)}:{360/\n * (13)}:\radius);
			
			\node at ({-360/\n * (14-4)}:1.2*\radius) {} ;
			\draw[-, >=latex] ({360/\n * (14 - 1)}:\radius) 
			arc ({360/\n * (14 -1)}:{360/\n * (14)}:\radius);
			
			\foreach \from/\to in {}
			\draw ({360/\n * (\from - 1)}:\radius) to [bend left=30] ({360/\n * (\to - 1)}:\radius);
			
			\foreach \from/\to in {14/6,2/9,4/11}
			\draw ({360/\n * (\from - 1)}:\radius) -- ({360/\n * (\to - 1)}:\radius); 
			\node[text width=6cm, anchor=west, right] at (-0.5,-2.5)
			{$R_1$};
			\end{tikzpicture}
		\end{center}
		
		Note that if $n$ is fixed, each three crossing matching of $F\cap R_1$ can be obtained by the following three steps:
		
		\begin{enumerate}
			\item Pick a subset $S\subset\{1,2,\dots,2n\}$ such that $|S|=n-3$ and $i\in S\Leftrightarrow i+\frac{2n}{3}\in S\Leftrightarrow i+\frac{4n}{3}\in S$ (in order to make sure the matching $\tau$ we get satisfies $\sigma_{2n}^{\frac{2n}{3}}(\tau)=\tau$).
			
			\item Apply $\NCC(n,S)$ and get a noncrossing partial matchings with exactly $6$ points unmatched.
			
			\item Then we have exactly one way to construct a $R_1$ type three crossing matching.
		\end{enumerate}
		
		So, we can divide the numbers $\{1,2,\dots,2n\}$ into $\frac{2n}{3}$ parts as the following:
		
		\begin{align*}
		& \phi_1=\{1, 1+\frac{2}{3}n,1+\frac{4}{3}n\},\\
		& \phi_2=\{2, 2+\frac{2}{3}n, 2+\frac{4}{3}n\},\\
		& \dots \\
		& \phi_{\frac{2n}{3}}=\{\frac{2}{3}n, \frac{4}{3}n, 2n\}.
		\end{align*}
		
		In step 1, picking $S$ is the same as picking $\frac{n-3}{3}$ parts from above. Hence, there are $\binom{\frac{2}{3}n}{\frac{n-3}{3}}$ ways to select such subset, which means 
		
		$$|F\cap R_1|=\binom{\frac{2}{3}n}{\frac{n-3}{3}}.$$
		
		Now, we look at other types of three crossing matchings. Let $R_2$ be the type of three crossing matchings contain the following partial matching.
		
		
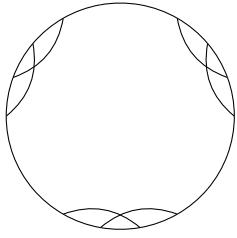
\begin{center}
		\begin{tikzpicture}
		
		\def \n {18}
		\def \radius {1.5cm}
		
		\node at ({-360/\n * (1-4)}:1.2*\radius) {} ;
		\draw[-, >=latex] ({360/\n * (1 - 1)}:\radius) 
		arc ({360/\n * (1 -1)}:{360/\n * (1)}:\radius);
		
		\node at ({-360/\n * (2-4)}:1.2*\radius) {} ;
		\draw[-, >=latex] ({360/\n * (2 - 1)}:\radius) 
		arc ({360/\n * (2 -1)}:{360/\n * (2)}:\radius);
		
		\node at ({-360/\n * (3-4)}:1.2*\radius) {} ;
		\draw[-, >=latex] ({360/\n * (3 - 1)}:\radius) 
		arc ({360/\n * (3 -1)}:{360/\n * (3)}:\radius);
		
		\node at ({-360/\n * (4-4)}:1.2*\radius) {} ;
		\draw[-, >=latex] ({360/\n * (4 - 1)}:\radius) 
		arc ({360/\n * (4 -1)}:{360/\n * (4)}:\radius);
		
		\node at ({-360/\n * (5-4)}:1.2*\radius) {} ;
		\draw[-, >=latex] ({360/\n * (5 - 1)}:\radius) 
		arc ({360/\n * (5 -1)}:{360/\n * (5)}:\radius);
		
		\node at ({-360/\n * (6-4)}:1.2*\radius) {} ;
		\draw[-, >=latex] ({360/\n * (6 - 1)}:\radius) 
		arc ({360/\n * (6 -1)}:{360/\n * (6)}:\radius);
		
		\node at ({-360/\n * (7-4)}:1.2*\radius) {} ;
		\draw[-, >=latex] ({360/\n * (7 - 1)}:\radius) 
		arc ({360/\n * (7 -1)}:{360/\n * (7)}:\radius);
		
		\node at ({-360/\n * (8-4)}:1.2*\radius) {} ;
		\draw[-, >=latex] ({360/\n * (8 - 1)}:\radius) 
		arc ({360/\n * (8 -1)}:{360/\n * (8)}:\radius);
		
		\node at ({-360/\n * (9-4)}:1.2*\radius) {} ;
		\draw[-, >=latex] ({360/\n * (9 - 1)}:\radius) 
		arc ({360/\n * (9 -1)}:{360/\n * (9)}:\radius);
		
		\node at ({-360/\n * (10-4)}:1.2*\radius) {} ;
		\draw[-, >=latex] ({360/\n * (10 - 1)}:\radius) 
		arc ({360/\n * (10 -1)}:{360/\n * (10)}:\radius);
		
		\node at ({-360/\n * (11-4)}:1.2*\radius) {} ;
		\draw[-, >=latex] ({360/\n * (11 - 1)}:\radius) 
		arc ({360/\n * (11 -1)}:{360/\n * (11)}:\radius);
		
		\node at ({-360/\n * (12-4)}:1.2*\radius) {} ;
		\draw[-, >=latex] ({360/\n * (12 - 1)}:\radius) 
		arc ({360/\n * (12 -1)}:{360/\n * (12)}:\radius);
		
		\node at ({-360/\n * (13-4)}:1.2*\radius) {} ;
		\draw[-, >=latex] ({360/\n * (13 - 1)}:\radius) 
		arc ({360/\n * (13 -1)}:{360/\n * (13)}:\radius);
		
		\node at ({-360/\n * (14-4)}:1.2*\radius) {} ;
		\draw[-, >=latex] ({360/\n * (14 - 1)}:\radius) 
		arc ({360/\n * (14 -1)}:{360/\n * (14)}:\radius);
		
		\node at ({-360/\n * (15-4)}:1.2*\radius) {} ;
		\draw[-, >=latex] ({360/\n * (15 - 1)}:\radius) 
		arc ({360/\n * (15 -1)}:{360/\n * (15)}:\radius);
		
		\node at ({-360/\n * (16-4)}:1.2*\radius) {} ;
		\draw[-, >=latex] ({360/\n * (16 - 1)}:\radius) 
		arc ({360/\n * (16 -1)}:{360/\n * (16)}:\radius);
		
		\node at ({-360/\n * (17-4)}:1.2*\radius) {} ;
		\draw[-, >=latex] ({360/\n * (17 - 1)}:\radius) 
		arc ({360/\n * (17 -1)}:{360/\n * (17)}:\radius);
		
		\node at ({-360/\n * (18-4)}:1.2*\radius) {} ;
		\draw[-, >=latex] ({360/\n * (18 - 1)}:\radius) 
		arc ({360/\n * (18 -1)}:{360/\n * (18)}:\radius);
		
		\foreach \from/\to in {1/3,2/4,7/9,8/10,13/15,14/16}
		\draw ({360/\n * (\from - 1)}:\radius) to [bend left=30] ({360/\n * (\to - 1)}:\radius);
		
		\foreach \from/\to in {}
		\draw ({360/\n * (\from - 1)}:\radius) -- ({360/\n * (\to - 1)}:\radius);
		\text{$ $} 
		\end{tikzpicture}
	\end{center}
	
	Then, for $2\leqslant k\leqslant \frac{n}{3}$, let $R_k$ be the type of three crossing matchings that contain the partial matching obtained by adding $k-2$ chords just outside each crossings in $R_2$, but do not contain the partial matching obtained by adding $k-1$ chords just outside each crossings in $R_2$. 
	
	For example, the type $R_3$ is the type of three crossing matchings that contain the partial matching obtain by adding $1$ chords just outside each crossings in $R_2$ (the partial matching $M_1$ below), but do not contain the partial matching obtain by adding $2$ chords just outside each crossings in $R_2$ (the partial matching $M_2$ below).
	
	
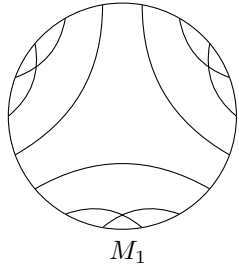
\begin{center}
		\begin{tikzpicture}
		
		\def \n {18}
		\def \radius {1.5cm}
		
		\node at ({-360/\n * (1-4)}:1.2*\radius) {} ;
		\draw[-, >=latex] ({360/\n * (1 - 1)}:\radius) 
		arc ({360/\n * (1 -1)}:{360/\n * (1)}:\radius);
		
		\node at ({-360/\n * (2-4)}:1.2*\radius) {} ;
		\draw[-, >=latex] ({360/\n * (2 - 1)}:\radius) 
		arc ({360/\n * (2 -1)}:{360/\n * (2)}:\radius);
		
		\node at ({-360/\n * (3-4)}:1.2*\radius) {} ;
		\draw[-, >=latex] ({360/\n * (3 - 1)}:\radius) 
		arc ({360/\n * (3 -1)}:{360/\n * (3)}:\radius);
		
		\node at ({-360/\n * (4-4)}:1.2*\radius) {} ;
		\draw[-, >=latex] ({360/\n * (4 - 1)}:\radius) 
		arc ({360/\n * (4 -1)}:{360/\n * (4)}:\radius);
		
		\node at ({-360/\n * (5-4)}:1.2*\radius) {} ;
		\draw[-, >=latex] ({360/\n * (5 - 1)}:\radius) 
		arc ({360/\n * (5 -1)}:{360/\n * (5)}:\radius);
		
		\node at ({-360/\n * (6-4)}:1.2*\radius) {} ;
		\draw[-, >=latex] ({360/\n * (6 - 1)}:\radius) 
		arc ({360/\n * (6 -1)}:{360/\n * (6)}:\radius);
		
		\node at ({-360/\n * (7-4)}:1.2*\radius) {} ;
		\draw[-, >=latex] ({360/\n * (7 - 1)}:\radius) 
		arc ({360/\n * (7 -1)}:{360/\n * (7)}:\radius);
		
		\node at ({-360/\n * (8-4)}:1.2*\radius) {} ;
		\draw[-, >=latex] ({360/\n * (8 - 1)}:\radius) 
		arc ({360/\n * (8 -1)}:{360/\n * (8)}:\radius);
		
		\node at ({-360/\n * (9-4)}:1.2*\radius) {} ;
		\draw[-, >=latex] ({360/\n * (9 - 1)}:\radius) 
		arc ({360/\n * (9 -1)}:{360/\n * (9)}:\radius);
		
		\node at ({-360/\n * (10-4)}:1.2*\radius) {} ;
		\draw[-, >=latex] ({360/\n * (10 - 1)}:\radius) 
		arc ({360/\n * (10 -1)}:{360/\n * (10)}:\radius);
		
		\node at ({-360/\n * (11-4)}:1.2*\radius) {} ;
		\draw[-, >=latex] ({360/\n * (11 - 1)}:\radius) 
		arc ({360/\n * (11 -1)}:{360/\n * (11)}:\radius);
		
		\node at ({-360/\n * (12-4)}:1.2*\radius) {} ;
		\draw[-, >=latex] ({360/\n * (12 - 1)}:\radius) 
		arc ({360/\n * (12 -1)}:{360/\n * (12)}:\radius);
		
		\node at ({-360/\n * (13-4)}:1.2*\radius) {} ;
		\draw[-, >=latex] ({360/\n * (13 - 1)}:\radius) 
		arc ({360/\n * (13 -1)}:{360/\n * (13)}:\radius);
		
		\node at ({-360/\n * (14-4)}:1.2*\radius) {} ;
		\draw[-, >=latex] ({360/\n * (14 - 1)}:\radius) 
		arc ({360/\n * (14 -1)}:{360/\n * (14)}:\radius);
		
		\node at ({-360/\n * (15-4)}:1.2*\radius) {} ;
		\draw[-, >=latex] ({360/\n * (15 - 1)}:\radius) 
		arc ({360/\n * (15 -1)}:{360/\n * (15)}:\radius);
		
		\node at ({-360/\n * (16-4)}:1.2*\radius) {} ;
		\draw[-, >=latex] ({360/\n * (16 - 1)}:\radius) 
		arc ({360/\n * (16 -1)}:{360/\n * (16)}:\radius);
		
		\node at ({-360/\n * (17-4)}:1.2*\radius) {} ;
		\draw[-, >=latex] ({360/\n * (17 - 1)}:\radius) 
		arc ({360/\n * (17 -1)}:{360/\n * (17)}:\radius);
		
		\node at ({-360/\n * (18-4)}:1.2*\radius) {} ;
		\draw[-, >=latex] ({360/\n * (18 - 1)}:\radius) 
		arc ({360/\n * (18 -1)}:{360/\n * (18)}:\radius);
		
		\foreach \from/\to in {1/3,2/4,7/9,8/10,13/15,14/16,18/5,6/11,12/17}
		\draw ({360/\n * (\from - 1)}:\radius) to [bend left=30] ({360/\n * (\to - 1)}:\radius);
		
		\foreach \from/\to in {}
		\draw ({360/\n * (\from - 1)}:\radius) -- ({360/\n * (\to - 1)}:\radius);
		\node[text width=6cm, anchor=west, right] at (-0.3,-1.8)
		{$M_1$};
		\end{tikzpicture}
	\end{center}
	
	
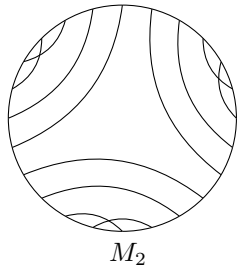
\begin{center}
		\begin{tikzpicture}
		
		\def \n {24}
		\def \radius {1.5cm}
		
		\node at ({-360/\n * (1-4)}:1.2*\radius) {} ;
		\draw[-, >=latex] ({360/\n * (1 - 1)}:\radius) 
		arc ({360/\n * (1 -1)}:{360/\n * (1)}:\radius);
		
		\node at ({-360/\n * (2-4)}:1.2*\radius) {} ;
		\draw[-, >=latex] ({360/\n * (2 - 1)}:\radius) 
		arc ({360/\n * (2 -1)}:{360/\n * (2)}:\radius);
		
		\node at ({-360/\n * (3-4)}:1.2*\radius) {} ;
		\draw[-, >=latex] ({360/\n * (3 - 1)}:\radius) 
		arc ({360/\n * (3 -1)}:{360/\n * (3)}:\radius);
		
		\node at ({-360/\n * (4-4)}:1.2*\radius) {} ;
		\draw[-, >=latex] ({360/\n * (4 - 1)}:\radius) 
		arc ({360/\n * (4 -1)}:{360/\n * (4)}:\radius);
		
		\node at ({-360/\n * (5-4)}:1.2*\radius) {} ;
		\draw[-, >=latex] ({360/\n * (5 - 1)}:\radius) 
		arc ({360/\n * (5 -1)}:{360/\n * (5)}:\radius);
		
		\node at ({-360/\n * (6-4)}:1.2*\radius) {} ;
		\draw[-, >=latex] ({360/\n * (6 - 1)}:\radius) 
		arc ({360/\n * (6 -1)}:{360/\n * (6)}:\radius);
		
		\node at ({-360/\n * (7-4)}:1.2*\radius) {} ;
		\draw[-, >=latex] ({360/\n * (7 - 1)}:\radius) 
		arc ({360/\n * (7 -1)}:{360/\n * (7)}:\radius);
		
		\node at ({-360/\n * (8-4)}:1.2*\radius) {} ;
		\draw[-, >=latex] ({360/\n * (8 - 1)}:\radius) 
		arc ({360/\n * (8 -1)}:{360/\n * (8)}:\radius);
		
		\node at ({-360/\n * (9-4)}:1.2*\radius) {} ;
		\draw[-, >=latex] ({360/\n * (9 - 1)}:\radius) 
		arc ({360/\n * (9 -1)}:{360/\n * (9)}:\radius);
		
		\node at ({-360/\n * (10-4)}:1.2*\radius) {} ;
		\draw[-, >=latex] ({360/\n * (10 - 1)}:\radius) 
		arc ({360/\n * (10 -1)}:{360/\n * (10)}:\radius);
		
		\node at ({-360/\n * (11-4)}:1.2*\radius) {} ;
		\draw[-, >=latex] ({360/\n * (11 - 1)}:\radius) 
		arc ({360/\n * (11 -1)}:{360/\n * (11)}:\radius);
		
		\node at ({-360/\n * (12-4)}:1.2*\radius) {} ;
		\draw[-, >=latex] ({360/\n * (12 - 1)}:\radius) 
		arc ({360/\n * (12 -1)}:{360/\n * (12)}:\radius);
		
		\node at ({-360/\n * (13-4)}:1.2*\radius) {} ;
		\draw[-, >=latex] ({360/\n * (13 - 1)}:\radius) 
		arc ({360/\n * (13 -1)}:{360/\n * (13)}:\radius);
		
		\node at ({-360/\n * (14-4)}:1.2*\radius) {} ;
		\draw[-, >=latex] ({360/\n * (14 - 1)}:\radius) 
		arc ({360/\n * (14 -1)}:{360/\n * (14)}:\radius);
		
		\node at ({-360/\n * (15-4)}:1.2*\radius) {} ;
		\draw[-, >=latex] ({360/\n * (15 - 1)}:\radius) 
		arc ({360/\n * (15 -1)}:{360/\n * (15)}:\radius);
		
		\node at ({-360/\n * (16-4)}:1.2*\radius) {} ;
		\draw[-, >=latex] ({360/\n * (16 - 1)}:\radius) 
		arc ({360/\n * (16 -1)}:{360/\n * (16)}:\radius);
		
		\node at ({-360/\n * (17-4)}:1.2*\radius) {} ;
		\draw[-, >=latex] ({360/\n * (17 - 1)}:\radius) 
		arc ({360/\n * (17 -1)}:{360/\n * (17)}:\radius);
		
		\node at ({-360/\n * (18-4)}:1.2*\radius) {} ;
		\draw[-, >=latex] ({360/\n * (18 - 1)}:\radius) 
		arc ({360/\n * (18 -1)}:{360/\n * (18)}:\radius);
		
		\node at ({-360/\n * (19-4)}:1.2*\radius) {} ;
		\draw[-, >=latex] ({360/\n * (19 - 1)}:\radius) 
		arc ({360/\n * (19 -1)}:{360/\n * (19)}:\radius);
		
		\node at ({-360/\n * (20-4)}:1.2*\radius) {} ;
		\draw[-, >=latex] ({360/\n * (20 - 1)}:\radius) 
		arc ({360/\n * (20 -1)}:{360/\n * (20)}:\radius);
		
		\node at ({-360/\n * (21-4)}:1.2*\radius) {} ;
		\draw[-, >=latex] ({360/\n * (21 - 1)}:\radius) 
		arc ({360/\n * (21 -1)}:{360/\n * (21)}:\radius);
		
		\node at ({-360/\n * (22-4)}:1.2*\radius) {} ;
		\draw[-, >=latex] ({360/\n * (22 - 1)}:\radius) 
		arc ({360/\n * (22 -1)}:{360/\n * (22)}:\radius);
		
		\node at ({-360/\n * (23-4)}:1.2*\radius) {} ;
		\draw[-, >=latex] ({360/\n * (23 - 1)}:\radius) 
		arc ({360/\n * (23 -1)}:{360/\n * (23)}:\radius);
		
		\node at ({-360/\n * (24-4)}:1.2*\radius) {} ;
		\draw[-, >=latex] ({360/\n * (24 - 1)}:\radius) 
		arc ({360/\n * (24 -1)}:{360/\n * (24)}:\radius);
		
		\foreach \from/\to in {1/3,2/4,9/11,10/12,17/19,18/20,24/5,23/6,8/13,7/14,16/21,15/22}
		\draw ({360/\n * (\from - 1)}:\radius) to [bend left=30] ({360/\n * (\to - 1)}:\radius);
		
		\foreach \from/\to in {}
		\draw ({360/\n * (\from - 1)}:\radius) -- ({360/\n * (\to - 1)}:\radius);
		\node[text width=6cm, anchor=west, right] at (-0.3,-1.8)
		{$M_2$};
		\end{tikzpicture}
	\end{center}
	
	Note that $\{F\cap R_1, F\cap R_2,\dots, F\cap R_{\frac{n}{3}}\}$ is a partition of $F$. We now prove the following lemma, which is useful to check \ref{one-third}.
	
	\begin{lemma}\thlabel{threecrossingpartition}
		$|F\cap R_k|=2k\binom{\frac{2}{3}n}{\frac{n-3k}{3}}$ for $2\leqslant k\leqslant \frac{n}{3}$.
	\end{lemma}
	\begin{proof}[Proof of \thref{threecrossingpartition}]
		We apply $\NCC$ algorithm again. For $2\leqslant k\leqslant \frac{n}{3}$, each three crossing matching of $F\cap R_k$ can be obtained by the following three steps:
		
		\begin{enumerate}
			\item Pick a subset $S\subseteq\{1,2,\dots,2n\}$ such that $|S|=n-3k$ and $i\in S\Leftrightarrow i+\frac{2n}{3}\in S\Leftrightarrow i+\frac{4n}{3}\in S$ (in order to make sure the matching $\tau$ we get satisfies $\sigma_{2n}^{\frac{2n}{3}}(\tau)=\tau$).
			
			\item Apply $\NCC(n,S)$ and get a noncrossing partial matchings with exactly $3k$ points unmatched.
			
			\item Then we have exactly $2k$ ways to construct a $R_k$ type three crossing matching.
		\end{enumerate}
		
		So there is a $2k$ to $1$ corresponding relation between the elements in $F\cap R_k$ and the subsets obtained by picking $\frac{n-3k}{3}$ of the $\phi_i$'s, where $\phi_i=\{i, i+\frac{2}{3}n, i+\frac{4}{3}n\}$ for $1\leqslant i\leqslant \frac{2n}{3}$.
		
		This implies that 
		
		$$|F\cap R_k|=2k\binom{\frac{2}{3}n}{\frac{n-3k}{3}}, \text{ for } 2\leqslant k\leqslant \frac{n}{3}.$$
	\end{proof}
	
	So, we can calculate $|F|$ by this partition and the combinatorial equalities $\sum_{j=0}^s \binom{s}{j}=2^s$ and $j\binom{s}{j}=s\binom{s-1}{j-1}$:
	
	\begin{align*}
	|F|  & = \sum_{k=1}^{\frac{n}{3}}|F\cap R_k|=\binom{\frac{2}{3}n}{\frac{n}{3}-1}+\sum_{k=2}^{\frac{n}{3}}2k\binom{\frac{2}{3}n}{\frac{n-3k}{3}}\\
	& =\binom{\frac{2}{3}n}{\frac{n}{3}-1}+\sum_{j=0}^{\frac{n}{3}-2}2(\frac{n}{3}-j)\binom{\frac{2}{3}n}{j} \text{ (let $j=\frac{n}{3}-k)$}\\
	& =\binom{\frac{2}{3}n}{\frac{n}{3}-1}+\frac{2}{3}n\sum_{j=0}^{\frac{n}{3}-2}\binom{\frac{2}{3}n}{j}-2\sum_{j=0}^{\frac{n}{3}-2}j\binom{\frac{2}{3}n}{j}\\
	& =\binom{\frac{2}{3}n}{\frac{n}{3}-1}+\frac{2}{3}n\sum_{j=0}^{\frac{n}{3}-2}\binom{\frac{2}{3}n}{j}-2\sum_{j=0}^{\frac{n}{3}-2}\frac{2}{3}n\binom{\frac{2}{3}n-1}{j-1}\\
	& =\binom{\frac{2}{3}n}{\frac{n}{3}-1}+\frac{2}{3}n\left[\frac{2^{\frac{2n}{3}}-\binom{\frac{2}{3}n}{\frac{n}{3}}}{2}-\binom{\frac{2}{3}n}{\frac{n}{3}-1}\right]-\frac{4}{3}n\left[\frac{2^{\frac{2}{3}n-1}}{2}-\binom{\frac{2}{3}n-1}{\frac{n}{3}-1}-\binom{\frac{2}{3}n-1}{\frac{n}{3}-2}\right]\\
	& =\binom{\frac{2}{3}n}{\frac{n}{3}-1}+\frac{n}{3}\left[-\binom{\frac{2}{3}n}{\frac{n}{3}} -2\binom{\frac{2}{3}n}{\frac{n}{3}-1}+4\binom{\frac{2}{3}n-1}{\frac{n}{3}-1}+4\binom{\frac{2}{3}n-1}{\frac{n}{3}-2}\right]\\
	& =\frac{n}{3}\binom{\frac{2}{3}n}{\frac{n}{3}-1}.
	\end{align*}
	
	Now, the only thing left to check is that $f_{n,3}=|F|$ if $3|n$. This can be done by noticing the following facts.
	
	\begin{fact}\thlabel{fact3}
		When $3\nmid m$, we have
		\begin{equation}
		[m]_q(u)=1+q+\cdots +q^{m-1}=\left\{
		\begin{array}{ll}
		1 & \text{ if } m\equiv 1\pmod 3\\
		1+u & \text{ if } m\equiv 2\pmod 3
		\end{array}
		\right.
		\end{equation}
		
		Consequently, for any two positive integers $m_1$ and $m_2$ such that $3\nmid m_1$ and $3\nmid m_2$, we have
		\begin{equation}
		\frac{[m_1]_q}{[m_2]_q}(u)=1.
		\end{equation}
	\end{fact}
	
	\begin{fact}\thlabel{fact4}
		When $3|m$,
		\begin{equation}
		[m]_q=(1+q+q^2)(1+q^3+q^6+\cdots +q^{m-3})
		\end{equation}
		Consequently, 
		\begin{equation}
		[m]_q(u)=0,
		\end{equation}
		and for any two positive integers $m_1$ and $m_2$ such that $3|m_1$ and $3|m_2$, we have
		\begin{equation}
		\frac{[m_1]_q}{[m_2]_q}(u)=\frac{1+u^3+u^6+\cdots +u^{m_1-3}}{1+u^3+u^6+\cdots +u^{m_2-3}}=\frac{m_1/3}{m_2/3}.
		\end{equation}
	\end{fact}
	
	Now, we calculate $f_{n,3}(u)$ for $3|n$. Note that 
	
	\begin{align*}
	f_{n,3}(q)& =\frac{1}{[3]_q} \binomial{n+5}{2}_q \binomial{2n}{n-4}_q+\binomial{2n}{n-3}_q\\
	& =\frac{[n+5]_q[n+4]_q}{[2]_q[1]_q}\cdot \frac{[2n]_q[2n-1]_q[2n-2]_q\cdots [n+6]_q[n+5]_q}{[3]_q[n-4]_q[n-5]_q\cdots [3]_q[2]_q}+\frac{[2n]_q[2n-1]_q\cdots [n+4]_q}{[n-3]_q[n-4]_q\cdots [1]_q}\\
	\end{align*}

	By using \thref{fact3} and \thref{fact4}, we have
	
	\begin{align*}
	f_{n,3}(u) & =\frac{[2n]_q}{[3]_q}\cdot \frac{[2n-3]_q}{[n-6]_q}\cdot \frac{[2n-6]_q}{[n-9]_q}\cdots \frac{[n+6]_q}{[3]_q}+\frac{[2n]_q[2n-3]_q\cdots [n+6]_q}{[n-3]_q[n-6]_q\cdots [3]_q}\\
	& =\frac{\frac{2n}{3}}{1}\cdot \frac{\frac{2n}{3}-1}{\frac{n}{3}-2}\cdot \frac{\frac{2n}{3}-2}{\frac{n}{3}-3}\cdots \frac{\frac{n}{3}+2}{1}+\frac{\frac{2n}{3}(\frac{2n}{3}-1)\cdots (\frac{n}{3}+2)}{(\frac{n}{3}-1)(\frac{n}{3}-2)\cdots 1}\\
	& =(\frac{n}{3}-1)\binom{\frac{2}{3}n}{\frac{n}{3}-1}+\binom{\frac{2}{3}n}{\frac{n}{3}-1}=\frac{n}{3}\binom{\frac{2}{3}n}{\frac{n}{3}-1}\\
	& = |\{\tau\in P_{n,3}|\sigma_{2n}^{\frac{2n}{3}}(\tau)=\tau\}| =|F|.
	\end{align*}
	
	Using the same approach, we can also check that $f_{n,3}(u^2)=|F|$ when $3|n$, and $f_{n,3}(u)=f_{n,3}(u^2)=0$ when $3\nmid n$, which finishes the proof of establishing the CSP of three crossing matchings. 
	
	\end{proof}
	
	\subsection{Matchings with More Number of Crossings}
	As we have proved the cyclic sieving phenomenon, it is natural to conjecture about such phenomenon for the matchings with more crossings. This opens the work of finding appropriate polynomials $f_{n,k}$ for $k \geqslant 4$ such that $(P_{n,k},f_{n,k},C_{2n})$ exhibits the cyclic sieving phenomenon.  
	
	\section{Acknowledgment}
	This paper is based on the work of 2016 summer REU program at the University of Michigan-Ann Arbor and some following work. The REU program was advised by Professor Thomas Lam. The authors would like to thank Professor Lam for many useful discussions and suggestions about this paper, and the financial support from the Department of Mathematics of the University of Michigan-Ann Arbor. The authors would also like to thank Professor Prasad Tetali for his suggestions about this paper.


\end{document}